\documentclass[12pt,reqno]{amsart}
\usepackage{amsmath, amsthm, ams
symb, amsopn, amsfonts, amstext, stmaryrd, enumerate, color, mathtools, 
mathrsfs, enumitem, url, accents}
\usepackage{graphicx,color}
\usepackage{hyperref}
\usepackage[margin=1.2in]{geometry}

\makeatletter
\newcommand{\opnorm}{\@ifstar\@opnorms\@opnorm}
\newcommand{\@opnorms}[1]{%
 \left|\mkern-1.5mu\left|\mkern-1.5mu\left|
 #1
 \right|\mkern-1.5mu\right|\mkern-1.5mu\right|
}
\newcommand{\@opnorm}[2][]{%
 \mathopen{#1|\mkern-1.5mu#1|\mkern-1.5mu#1|}
 #2
 \mathclose{#1|\mkern-1.5mu#1|\mkern-1.5mu#1|}
}
\makeatother

\theoremstyle{plain}

\begin{document}


\theoremstyle{plain}
\newtheorem{theorem}{Theorem} [section]
\newtheorem{corollary}[theorem]{Corollary}
\newtheorem{lemma}[theorem]{Lemma}
\newtheorem{proposition}[theorem]{Proposition}


\theoremstyle{definition}
\newtheorem{definition}[theorem]{Definition}
\theoremstyle{remark}
\newtheorem{remark}[theorem]{Remark}

\numberwithin{theorem}{section}
\numberwithin{equation}{section}

\title[Optimal embeddings for Lebesgue and Morrey norms]{The optimal exponent in the embedding into the Lebesgue spaces for functions with gradient in the Morrey space}

\thanks{The authors were supported by MINECO grant MTM2014-52402-C3-1-P.
X. Cabr\'e is also supported by MINECO grants MTM2017-84214-C2-1-P and  MDM-2014-0445, and is member of the Catalan research group 2017
SGR 1392.
F. Charro was also partially supported by a Juan de la Cierva fellowship.
}

\author[X. Cabr\'e]{Xavier Cabr\'e}
\address{X.C.\textsuperscript{1,2,3} ---
\textsuperscript{1}ICREA, Pg.\ Lluis Companys 23, 08010 Barcelona, Spain \&
\textsuperscript{2}Universitat Polit\`ecnica de Catalunya, Departament de Matem\`{a}tiques,
Diagonal 647, 08028 Barcelona, Spain \&
\textsuperscript{3}BGSMath, Campus de Bellaterra, Edifici C, 08193 Bellaterra, Spain.
}
\email{xavier.cabre@upc.edu}

\author[F. Charro]{Fernando Charro}
\address{F.C. ---
Department of Mathematics, Wayne State  University, 656 W. Kirby, Detroit, MI 48202, USA.}
\email{fcharro@wayne.edu}

\keywords{Morrey spaces, Optimal embeddings, Cantor sets.
\\
\indent 2010 {\it Mathematics Subject Classification:}
42B37, 46E35}
\date{}

\begin{abstract}
We study the following natural question that, apparently, has not been well addressed in the literature:
Given  functions $u$ with support in the unit ball 
$B_1\subset\mathbb{R}^n$ and with
gradient in the Morrey space $M^{p,\lambda}(B_1)$, where $1<p<\lambda<n$, what is the largest range of  exponents $q$ for which necessarily  
$u\in L^{q}(B_1)$?
While David R. Adams proved in 1975 that this embedding holds for $q\leq\lambda p/(\lambda-p)$, an article from 2011 claimed the embedding in the larger range $q<n p/(\lambda-p)$. Here we disprove this last statement by constructing  a function that provides a counterexample for $q>\lambda p/(\lambda-p)$. The function is basically a negative power of the distance to a set of Hausdorff dimension $n-\lambda$. When $\lambda\notin\mathbb{Z}$, this set is a fractal.  We also make a detailed study of the radially symmetric case, a situation in which the exponent $q$ can go up to $np/(\lambda-p)$.
\end{abstract}

\maketitle


\section{Introduction}\label{section.intro}

This article originated from the following natural question:
{\it Given functions $u$ with support in the unit ball $B_1(0)\subset\mathbb{R}^n$ and with
gradient in the Morrey space $M^{p,\lambda}(B_1(0))$, where $1<p<\lambda<n$,
what is the  optimal range of exponents $q$ such that necessarily $u\in L^{q}(B_1(0))$?}
Apparently, this question has not been well addressed in the literature. 
In fact, the authors of \cite[Theorem 2.5]{Adams.Xiao.2011} claimed a range of exponents which, as we will prove in the current paper, turns out to be larger than the true one.

Our motivation came from the recent work \cite{Cabre} of the first
author in collaboration with A. Figalli, X. Ros-Oton, and J. Serra, on the regularity of
stable solutions to semilinear elliptic equations. 
Actually, the results of \cite{Cabre} are deduced from a Morrey type bound for the 
gradient of a stable solution, among other tools (see Remark \ref{more.details.on.paper} below for more details).

The following is the precise statement of the question that we are concerned with. Given real numbers $p$ and $\lambda$ such that
\[
1< p<\lambda<n,
\]
we wish to know for which   exponents $q$  the inequality 
\begin{equation}\label{estimate.intro.morrey}
\|u\|_{L^{q}(B_1(0))}\leq C\, \|\nabla u\|_{M^{p,\lambda}(B_1(0))}
\end{equation}
holds true for functions $u$ with support in $B_1(0)\subset\mathbb{R}^n$ and for a constant $C$ independent of $u$, where
\[
\|\nabla u\|_{M^{p,\lambda} (\Omega)}^{p}
:=
\sup_{r>0,\,y\in \overline\Omega} \bigg(r^{\lambda-n} \int_{\Omega\,\cap B_r(y)} |\nabla u(x)|^{p}\, dx\bigg)
\]
 is the Morrey norm of $\nabla u$ in a domain $\Omega\subset\mathbb{R}^n$.
Notice that when $\lambda$ equals the dimension $n$ and $\Omega=B_1(0)$, \eqref{estimate.intro.morrey} corresponds to the Sobolev inequality in $B_1(0)\subset\mathbb{R}^n$.

In 1975, D. R. Adams \cite[Theorem 3.1]{Adams.1975} proved the following result. 
Let us denote in the sequel
\[
 p_1:=\frac{\lambda p}{\lambda-p}\qquad\textrm{and}\qquad p_2:=\frac{n p}{\lambda-p}.
\]
Observe that, clearly, $ p_1< p_2$.

\begin{theorem}[D. R. Adams \cite{Adams.1975}]\label{Adams.original.result}
Let $p,\lambda\in\mathbb{R}$ satisfy $1<p<\lambda<n$ and let  $u:\mathbb{R}^n\to\mathbb{R}$ be a Lipschitz function with $u\equiv0$ in $\mathbb{R}^n\setminus B_1(0)$.
 Then, for every $q\leq p_1$, inequality \eqref{estimate.intro.morrey} holds for a constant $C$ depending only on $n, p,$ and $\lambda$.
\end{theorem}

For the reader's convenience, in Section \ref{section.embeddings.general} we will include his proof in this case $p>1$. 
The case $p=1$, which involves weak spaces, is also treated in \cite{Adams.1975}.

In fact, Adams \cite[Proposition 3.1 and Theorems 3.1 and 3.2]{Adams.1975} proved the following stronger embedding:
\begin{equation}\label{Adams.scale.invariant}
\|u\|_{M^{p_1,\lambda}(B_1(0))}\leq C \|\nabla u\|_{M^{p,\lambda}(B_1(0))}.
\end{equation}
While  inequality \eqref{Adams.scale.invariant} is dimensionless by scaling, note that the dimensionless exponent for inequality \eqref{estimate.intro.morrey} is $ q=p_2$. This suggests that \eqref{estimate.intro.morrey} could hold with $q=p_2$, or at least for $q<p_2$.
In fact, it is easy to prove that among radially symmetric functions, \eqref{estimate.intro.morrey} holds for every $q< p_2$ (see \cite[Proposition 1.2(i)]{Adams.Xiao.2016} and also Theorem \ref{main.thm.radial} below).

In 2011, the authors of \cite[Theorem 2.5]{Adams.Xiao.2011} claimed that \eqref{estimate.intro.morrey} held for every $q< p_2$ and for 
general functions,  not necessarily radial. Some years later, we realized that the proof of \cite[Theorem 2.5]{Adams.Xiao.2011} was not correct. After that, the claim was withdrawn by the same authors in the Errata papers \cite{Adams.Xiao.ERR.2015} and \cite{Adams.Xiao.2016}. 
At the same time, we could not find other works addressing the exact question of what is the optimal exponent.

In the present paper we  show that actually $q= p_1$ is the largest possible exponent in \eqref{estimate.intro.morrey}. 
To show this,  for every $q> p_1$ we construct a non-radial function, described  in detail below, for which 
$\|u\|_{L^{q}(B_1(0))}=\infty$ while $\|\nabla u\|_{M^{p,\lambda}(B_1(0))}<\infty$. 
An important feature of the function is that it depends only (up to a cutoff function) on $k$ variables $(x_1,\ldots,x_k)$, 
where $k$ is the smallest integer such that $\lambda\leq k$.  
The function is basically a negative power of the distance to a set of 
Hausdorff dimension $n-\lambda$. When $\lambda$ is not an integer this set is a fractal and, therefore, the structure of the 
function is not so ``simple''. In fact, it will be rather delicate to control the Morrey norm of its gradient.
We could not find a simpler counterexample 
for $\lambda\notin\mathbb{Z}$, although we had several candidates that finally did not work. 
The possibility of finding simpler examples remains as an open question.

In addition, we also consider a related norm, which we call the ``triple norm", given~by
\begin{equation}\label{triple}
\opnorm{\nabla u}_{p,\lambda;\Omega}^{p} :=\sup_{y\in \overline\Omega}\int_{\Omega} |\nabla u(x)|^{p} \,|x-y|^{\lambda-n}\, dx,
 \end{equation}
where $\Omega\subset\mathbb{R}^n$ is a domain. The article \cite{Cabre} on stable solutions to semilinear 
equations gives rise naturally to such a norm (see Remark \ref{more.details.on.paper} below).\footnote{The triple norm has been previously considered in the setting of the hole-filling technique for integral estimates; see \cite[Section 1.2.3]{Bensoussan-Frehse} among others. It also appears in \cite{Hou-Xiao}, where it is called the Cordes-Nirenberg norm.
}
Note that, clearly, we have
 \begin{equation}\label{relation.norms}
\|\nabla u\|_{M^{p,\lambda} (\Omega)} \leq \opnorm{\nabla u}_{p,\lambda;\Omega}
\end{equation}
for every function $u$.
The function that we construct will also satisfy $\opnorm{\nabla u}_{p,\lambda;B_1(0)}<~\infty$, and thus $q= p_1$ is the largest possible exponent also for the embedding
\begin{equation}\label{estimate.intro.triple}
\|u\|_{L^{q}(B_1(0))}\leq C\, \opnorm{\nabla u}_{p,\lambda;B_1(0)}.
\end{equation}
Instead, among radial functions we show that inequality \eqref{estimate.intro.triple} holds for every $q\leq p_2$, in contrast  to inequality \eqref{estimate.intro.morrey} for radial functions, which holds only for 
 $q< p_2$.

\begin{remark}\label{more.details.on.paper}
 The regularity results from the recent paper \cite{Cabre} on stable solutions to semilinear equations $-\Delta u= f(u)$ in a domain $\Omega\subset \mathbb{R}^n$ are based on bounds for a Morrey norm with $p=2$, 
 of $\nabla u$ or, given a point $y$, of the radial derivative $\nabla u(x)\cdot(x-y)/|x-y|$. For this, see
 \cite[Lemma 2.1, step 2 in the proof of Theorem 1.2, and proof of Theorem 7.1]{Cabre}, which also lead to the triple norm \eqref{triple}. 
 The boundedness results from \cite{Cabre} up to dimension $n\leq9$ correspond to $p=2$ and $\lambda=2$, while in the $L^q$ results for 
 $n\geq 11$ one has $p=2<\lambda<n$ as in our paper. The results of the current article are used in \cite{Cabre} to determine optimally a 
 range of exponents $q$ for which stable solutions necessarily belong to $L^q$ in dimensions $n\geq11$.
\end{remark}

Summarizing, our main contribution is the following result. It provides a counterexample to the validity of \eqref{estimate.intro.morrey} 
and \eqref{estimate.intro.triple} for $q>p_1$, given by a function $u$ which is basically a negative power of the distance to a set of 
Hausdorff dimension $n-\lambda$. When $\lambda\notin\mathbb{Z}$, this set is a fractal.

\begin{theorem}\label{intro.counterexamples}
Let $p,\lambda\in\mathbb{R}$ satisfy $1<p<\lambda<n$. Then, for every $q> p_1:=\lambda p/(\lambda-p)$ there exists a function $u:\mathbb{R}^n\to\mathbb{R}$ with $u\equiv0$ in $\mathbb{R}^n\setminus B_1(0)$, 
\begin{equation}\label{counterexamples.properties.thm}
\|u\|_{L^{q}(B_1(0))}=\infty,\quad\textrm{and}\quad \opnorm{\nabla u}_{p,\lambda;B_1(0)}<\infty.
\end{equation}
In particular, we also have
\[
\|\nabla u\|_{M^{p,\lambda} (B_1(0))} <\infty.
\]

If $\lambda$ is an integer, such function $u$ can be taken to be
\begin{equation}\label{counterexample.noninteger.general.statement.thm}
u(x)=\left(|x'|^{-\alpha}-2^\alpha\right)_{+}\,\xi(|x''|) 
\end{equation}
where $x=(x',x'')\in\mathbb{R}^{\lambda}\times\mathbb{R}^{n-\lambda}$, 
 the parameter $\alpha$ satisfies
\begin{equation}\label{hyp.alpha.counterexamples.intro.statement.thm}
\frac{\lambda}{q}\leq\alpha<\frac{\lambda-p}{p},
\end{equation}
and $\xi:\mathbb{R}^+\to[0,1]$ is a cutoff function with $\xi\equiv 1$ in $[0,1/2)$ and $\xi\equiv0$ in $\mathbb{R}^+\setminus[0,\sqrt{3}/3).$

If $k-1<\lambda<k$ for some integer $k\in[2,n]$, the function $u$ can be taken to be
\begin{equation}\label{counterexample.noninteger.intro.statement.thm}
u(x) =
\left\{
\begin{split}
&\left(\textnormal{dist}(x,\mathcal{C}_{n,\lambda})^{-\alpha}-4^\alpha\right)_{+}\hspace{62pt}\textnormal{if}\ k=n \\
&\left(\textnormal{dist}(x',\mathcal{C}_{k,\lambda})^{-\alpha}-4^\alpha\right)_{+}\,\xi(|x''|) \qquad\textnormal{if}\ k<n,
\end{split}
\right.
\end{equation}
where $x=(x',x'')\in\mathbb{R}^{k}\times\mathbb{R}^{n-k}$, $\alpha$ satisfies \eqref{hyp.alpha.counterexamples.intro.statement.thm}, $\xi$ is a cutoff function as above, and $\mathcal{C}_{k,\lambda}$ is a set of Hausdorff dimension $k-\lambda$ in $\mathbb{R}^k$ given by
\[
\mathcal{C}_{k,\lambda}=\{0\}\times C_{\gamma}\subset\mathbb{R}^{k-1}\times[-1/2,1/2],
\]
where $C_{\gamma}\subset[-1/2,1/2]$ is the generalized Cantor set with parameter 
$\gamma=1-2^{1-\frac{1}{k-\lambda}}$  defined in the following remark.
\end{theorem}
 
 We emphasize that the counterexample to the embedding is therefore given by a function that, up to a cutoff, only depends on $k$ variables, where $k$ is the smallest integer such that $\lambda\leq k$.

\begin{remark}\label{remark.introduction.Cantor.sets}
 The generalized Cantor set $C_{\gamma}$ (see \cite{Falconer}) is obtained from the interval $[-1/2,1/2]$ by removing at iteration $j=1,2,\ldots$ the central interval of length $\gamma\,l_{j-1}$ from each remaining segment of length $l_{j-1}=((1-\gamma)/2)^{j-1}$; 
see Figure \ref{figure.gencantor} in Section~\ref{section.lambda.noninteger}. The usual Cantor set corresponds to $\gamma=1/3$.
The reason for our choice of $\gamma$ is that
the Hausdorff dimension of $C_{\gamma}$ is 
\[
\frac{-\log{2}}{\log{\frac{1-\gamma}{2}}}=k-\lambda\in(0,1)
\]
 (see \cite[Theorem 9.3]{Falconer}).
 In particular, 
 letting $\lambda$ range from $k-1$ to $k$ yields any fractal dimension between 0 and 1, and \eqref{counterexample.noninteger.intro.statement.thm} somehow interpolates between the integer cases $\lambda=k-1$ and $\lambda=k$.
\end{remark}

Let us describe briefly how we found that $p_1$ is the optimal exponent.
In the case when $\lambda$ is an integer, the hint came from the number $ p_1=\lambda p/(\lambda-p)$, which can be thought of as the Sobolev exponent in dimension $\lambda$. It was natural then to choose the function \eqref{counterexample.noninteger.general.statement.thm}, since it gives a counterexample for the Sobolev inequality in $\mathbb{R}^\lambda$ when
$q>\lambda p/(\lambda-p)=p_1$ and the exponent $\alpha$ is chosen appropriately.

When $\lambda\in\mathbb{Z}$, \eqref{counterexample.noninteger.general.statement.thm} is basically a negative power of the distance to a subspace of dimension $n-\lambda$.
Therefore, when $\lambda\notin\mathbb{Z}$, 
 a negative power of the distance to a set of Hausdorff dimension $n-\lambda$ became a natural candidate
to counterexample. 
 This is what the function in \eqref{counterexample.noninteger.intro.statement.thm} basically is, 
 a power of the distance to $\mathcal{C}_{k,\lambda}\times \mathbb{R}^{n-k}$.
 
It may be of interest to recall here the solutions 
found by R. Schoen and S.-T. Yau in~\cite[Section 5]{Schoen-Yau} for nonlinear equations with critical exponent. 
They construct weak solutions which are singular on a Cantor set with fractional Hausdorff dimension; 
see~\cite[Page 70]{Schoen-Yau}. Obviously, nonlinear equations with critical exponent 
are closely related to the Sobolev embedding. 
Another result on solutions with a singular set of Cantor type is due to Fonseca, Mal\'y, and 
Mingione~\cite{Fonseca-Maly-Mingione}, a paper that concerns the minimizers of a certain
scalar, convex, and regular Lagrangian.

The paper \cite{Adams.Lewis.1982} by Adams and Lewis was brought to our attention after the completion
of the current article. In~ \cite{Adams.Lewis.1982}, the authors proved that functions 
which satisfy an integrability condition of Morrey-Besov type belong also to a certain 
Lorentz space. In addition, they construct examples of functions to show that their embeddings
are the best possible.
The Morrey-Besov norm is  a fractional  Morrey-type condition involving the $\alpha$-th difference quotients 
of a function, where $0<\alpha<1$. Thus, this concerns more exotic norms than the 
basic and standard ones that we treat. The authors of~\cite{Adams.Lewis.1982} mention the possibility that the ideas  
in their proof of \cite[Theorem 3]{Adams.Lewis.1982} could be extended from the case $\alpha\in(0,1)$ 
that they treat to the case $\alpha=1$ (and $p=q$ in their paper). This is something that we have not explored.
There could be the usual delicate issues taking limits of fractional integral norms as $\alpha\to 1$, or 
even simply the impossibility of taking this limit or adapting the proof for $\alpha=1$.
However, if this could be done, it would show that  actually $q= p_1$ is the largest possible exponent 
in \eqref{estimate.intro.morrey}. On the other hand, it is not 
clear at all if their example  would allow to recover our result on the optimal exponent for the embedding  
\eqref{estimate.intro.triple} concerning the ``triple'' norm.

Among radial functions, the optimal ranges of exponents in inequalities \eqref{estimate.intro.morrey} and \eqref{estimate.intro.triple} are strictly larger
than those of Theorems \ref{Adams.original.result} and \ref{intro.counterexamples}. 
This is the content of the following result, where we show that the exponent $q$ can go up to $q_2$. 
Interestingly, here the answer is different for the Morrey and the ``triple" norms: 
we prove that \eqref{estimate.intro.morrey} is false for $q= p_2$ while \eqref{estimate.intro.triple}
holds for this exponent. Here we can include the exponent $p=1$.

\begin{theorem}\label{main.thm.radial}
 Let $p,\lambda\in\mathbb{R}$ satisfy $1\leq p<\lambda<n$, and let $p_2:=n p/(\lambda-p)$. 
\begin{enumerate}

 \item[(a)] For every $1\leq q<p_2$ and all radially symmetric 
 ${C}^1$ functions $u$ vanishing on $\partial B_1(0)$, we have
 \begin{equation}\label{lemma.radial.estimate.Morrey}
\|u\|_{L^{q}(B_1(0))}\leq C\, \|\nabla u\|_{M^{p,\lambda}(B_1(0))},
\end{equation}
where $C$ is a constant depending only on $n,p,\lambda$, and $q$. In addition, this embedding is false for $q\geq p_2$.

 \item[(b)] For all radially symmetric ${C}^1$ functions $u$ with compact support in $\mathbb{R}^n$, we have
 \begin{equation}\label{theorem.radial.estimate.triple}
\|u\|_{L^{p_2}(\mathbb{R}^n)}\leq C\, \opnorm{\nabla u}_{p,\lambda;\mathbb{R}^n},
\end{equation}
where $C$ is a constant depending only on $n,p,$ and $\lambda$. In addition, 
 $ p_2$ is the optimal exponent in this inequality.

%

 \end{enumerate}
\end{theorem}

The paper is organized as follows.
{In Section \ref{section.monotonicity}, we prove a monotonicity result, Lemma \ref{lemma.monotonicity}, that 
 we will use several times throughout the paper to optimize the location of the ``singularities" $y$ in the Morrey and ``triple" norms, both in the radial and non-radial cases.} In Section \ref{section.radial} we prove the embeddings for radial functions, Theorem \ref{main.thm.radial}. 
In Section \ref{section.embeddings.general} we provide for the reader's convenience D. R. Adams' \cite{Adams.1975} proof of Theorem \ref{Adams.original.result} in the general case of non-radial functions. In Sections \ref{section.lambda.integer} and \ref{section.lambda.noninteger} we prove Theorem \ref{intro.counterexamples} on the optimality of the embeddings. 
We consider separately the case when $\lambda$ is an integer in Section \ref{section.lambda.integer} (for its simplicity) and the case when $\lambda$ is a non-integer in Section \ref{section.lambda.noninteger} (which is much more involved).

\noindent{\bf Notation.}
In the sequel $B_R^{(m)}(x)$ denotes the  open ball in $\mathbb{R}^m$ of radius $R$ centered at~$x$. For simplicity, whenever $m$ or $x$ are omitted, we will consider $m=n$ and $x=0$ respectively. By $C$, we denote  constants that may change from line to line.
For points in $\mathbb{R}^n$, we will write $x=(x',x'')\in\mathbb{R}^{k}\times\mathbb{R}^{n-k}$ for $k$ a positive integer specified from the context.
Given a function $u$, $u_+=\max\{u,0\}$ is its positive part.  
As mentioned before, we will denote $ p_1=\lambda p/(\lambda-p)$ and $ p_2=n p/(\lambda-p)$.  For convenience,  we will use the following standard notation  for intervals: $a(-b,b)+h=(h-ab,h+ab)$.
Finally, $\textrm{dist}\left(t, U\right)=\inf_{z\in U}|t-z|$ as usual.



\section{On the location of the singularity in the ``triple" norm}\label{section.monotonicity}
{
In this section we prove a monotonicity result that we will use several times in the sequel to study which locations of the ``singularity" $y$ make larger the integral in the  ``triple norm"
\[
\opnorm{\nabla u}_{p,\lambda;\Omega}^{p} =\sup_{y\in \overline\Omega}\int_{\Omega} |\nabla u(x)|^{p} \,|x-y|^{\lambda-n}\, dx.
\]

\begin{lemma}\label{lemma.monotonicity}
Consider a domain $\Omega\subset\mathbb{R}^n$, convex in the $e_1$ direction, and symmetric with respect to $\{z_1=0\}$. 
Let $J:\mathbb{R}^n\to\mathbb{R}$ be given by
\[
 J(y):=
\int_{\Omega} 
h(z)|z-{y}|^{-\theta}\, dz
\]
with $\theta>0$ and $h$ a non-negative function in $\Omega$. Then:

\begin{enumerate}

\item[(a)] $J $ is non-increasing with respect to ${y}_1$ in $\{y_1\geq \sup_{z\in\Omega}z_1\}$, and non-decreasing with respect to ${y}_1$ in $\{y_1\leq \inf_{z\in\Omega}z_1\}$.

 \item[(b1)] Suppose that the non-negative function $h$ satisfies that, 
 for some $\eta\in[0,\sup_{z\in\Omega}z_1)$ and every $y_1\in[\eta,\sup_{z\in\Omega}z_1)$, 
\[
h(z^*)\geq h(z)\quad\textrm{for all $z\in\Omega\cap\{z_1\geq y_1 \}$},
\]
where $z^*=(2y_1-z_1,z')\in\mathbb{R}\times\mathbb{R}^{n-1}$ is the reflection of $z$ with respect to the hyperplane $\{z_1={y_1}\}$.
Then, $J $ is non-increasing with respect to ${y}_1$ in $\{\eta\leq y_1< \sup_{z\in\Omega}z_1\}$.

 \item[(b2)] On the other hand, if the non-negative function $h$ is such that,
 for some $\eta\in(\inf_{z\in\Omega}z_1,0]$ and every $y_1\in(\inf_{z\in\Omega}z_1,\eta]$, 
\[
h(z^*)\geq h(z)\quad\textrm{for all $z\in\Omega\cap\{z_1\leq y_1 \}$}
\]
with $z^*$ as before, then $J $ is non-decreasing with respect to ${y}_1$ in $\{\inf_{z\in\Omega} z_1<y_1\leq \eta\}$.

\end{enumerate}
\end{lemma}

\begin{proof}
Observe first that for every $z\in\Omega$, the quantity $|z-{y}|$ is 
increasing with respect to ${y}_1$ in $\{y_1\geq \sup_{z\in\Omega}z_1\}$, and decreasing with respect to ${y}_1$ in $\{y_1\leq \inf_{z\in\Omega}z_1\}$.
 Since $h\geq0$ and $\theta>0$, we deduce that $J$ is non-increasing with respect to ${y}_1$ in $\{y_1\geq \sup_{z\in\Omega}z_1\}$, and non-decreasing with respect to ${y}_1$ in $\{y_1\leq \inf_{z\in\Omega}z_1\}$. This proves part (a).

\begin{figure}[t]
\setlength{\unitlength}{1.3cm}
\begin{picture}(6,6.5)(-3,-3)
\put(0,0){\vector(1,0){3}}
\put(0,0){\line(-1,0){3}}
\put(0,0){\vector(0,1){3}}
\put(0,0){\line(0,-1){3}}

\qbezier(0.00,2.50)(0.50,2.50)(0.96,2.31)
\qbezier(0.96,2.31)(1.42,2.12)(1.77,1.77)
\qbezier(1.77,1.77)(2.12,1.42)(2.31,0.96)
\qbezier(2.31,0.96)(2.50,0.50)(2.50,0.00)
\qbezier(-0.00,2.50)(-0.50,2.50)(-0.96,2.31)
\qbezier(-0.96,2.31)(-1.42,2.12)(-1.77,1.77)
\qbezier(-1.77,1.77)(-2.12,1.42)(-2.31,0.96)
\qbezier(-2.31,0.96)(-2.50,0.50)(-2.50,0.00)
\qbezier(0.00,-2.50)(0.50,-2.50)(0.96,-2.31)
\qbezier(0.96,-2.31)(1.42,-2.12)(1.77,-1.77)
\qbezier(1.77,-1.77)(2.12,-1.42)(2.31,-0.96)
\qbezier(2.31,-0.96)(2.50,-0.50)(2.50,-0.00)
\qbezier(-0.00,-2.50)(-0.50,-2.50)(-0.96,-2.31)
\qbezier(-0.96,-2.31)(-1.42,-2.12)(-1.77,-1.77)
\qbezier(-1.77,-1.77)(-2.12,-1.42)(-2.31,-0.96)
\qbezier(-2.31,-0.96)(-2.50,-0.50)(-2.50,-0.00)

\qbezier(1.77,1.77)(1.42,1.42)(1.23,0.96)
\qbezier(1.23,0.96)(1.04,0.50)(1.04,0.00)
\qbezier(1.77,-1.77)(1.42,-1.42)(1.23,-0.96)
\qbezier(1.23,-0.96)(1.04,-0.50)(1.04,0.00)

\multiput(1.77,-1.85)(0,.35){11}{\line(0,1){.2}} 

\put(1.77,0){\line(1,3){.3}}
\put(1.77,0){\line(-1,3){.3}}
\put(1.47,.9){\line(1,0){.6}}

\put(0,0){\circle*{.1}} 
\put(2.07,0.9){\circle*{.1}} 
\put(1.47,0.9){\circle*{.1}} 
\put(1.77,0){\circle*{.1}} 

\put(-.3,-.3){$0$}
\put(3,-.3){$z_1$}
\put(1.35,1.04){$z^*$}
\put(2,1.04){$z$}
\put(1.46,-.3){$y_1$}

\end{picture}
\caption{Monotonicity argument in the proof of Lemma \ref{lemma.monotonicity}.}
\label{figure.lemma.monotonicity.radial}
\end{figure}

Assume now that $0\leq\eta\leq{y}_1<\sup_{z\in\Omega}z_1$ and 
compute
\[
\begin{split}
\partial_{{y}_1}J(y)
&=\theta
\int_{\Omega} 
h(z)(z_1-{y}_1)|z-{y}|^{-\theta-2}
\, dz
\\
&=\theta
\Bigg\{
\int_{\Omega\cap\{z_1\geq y_1 \}}
h(z)(z_1-{y}_1)|z-{y}|^{-\theta-2}
 \, dz
\\
&\qquad+
\int_{\Omega\cap\{z_1\leq y_1 \}}
h(z)(z_1-{y}_1)|z-{y}|^{-\theta-2}
 \, dz
 \Bigg\}.
\end{split}
\]
For every $z\in\Omega\cap\{z_1\geq y_1 \}$, let $z^*=(2{y}_1-z_1,z')$ be its reflection with respect to the hyperplane $\{z_1={y}_1\}$; 
see Figure \ref{figure.lemma.monotonicity.radial}. 
Then, $|z^*-{y}|=|z-{y}|$, while $h(z^*)\geq h(z)$ by hypothesis. Therefore, for every $z\in\Omega\cap\{z_1\geq y_1 \}$, we have
\[
h(z)(z_1-{y}_1)|z-{y}|^{-\theta-2}
 \leq
- h(z^*)(z_1^*-{y}_1)|z^*-{y}|^{-\theta-2}
\]
and hence, using that $h\geq0$,
\[
\begin{split}
\int_{\Omega\cap\{z_1\geq y_1 \}}&\,
h(z)(z_1-{y}_1)|z-{y}|^{-\theta-2}\, dz
 \\
&\leq
-\int_{\left(\Omega\cap\{z_1\geq y_1 \}\right)^*}
h(z)(z_1-{y}_1)|z-{y}|^{-\theta-2}\, dz
 \\
&\leq - \int_{\Omega\cap\{z_1\leq y_1 \}}
h(z)(z_1-{y}_1)|z-{y}|^{-\theta-2}\, dz.
\end{split}
\]
Therefore, $\partial_{{y}_1}J(y)\leq0$ for all $\eta\leq{y}_1<\sup_{z\in\Omega}z_1$ and the conclusion in part (b1) follows.  The statement for $\inf_{z\in\Omega}z_1<{y}_1\leq\eta$ in part (b2) follows from (b1) by reflection.
\end{proof}


\section{The radial case: Proof of Theorem \ref{main.thm.radial}}\label{section.radial}

In this section we  establish Theorem \ref{main.thm.radial} on the embeddings for radial functions. In the proof we apply Lemma \ref{lemma.monotonicity} (the monotonicity result proved in Section \ref{section.monotonicity}), which will also be used for the non-radial case in Section \ref{section.lambda.noninteger}. We point out that the Sobolev inequalities with monomial weights established in \cite{Cabre.Ros-Oton} by Ros-Oton and the first author will be of great use.

\begin{proof}[Proof of Theorem \ref{main.thm.radial}]
We structure the proof in four parts. In Part 1 we establish estimate \eqref{lemma.radial.estimate.Morrey}, while  in Part 2 we show that $q< p_2$ is the optimal range of exponents for this estimate.
In Part 3a we prove \eqref{theorem.radial.estimate.triple};
 here we will use the results of \cite{Cabre.Ros-Oton}. Part 3b provides an alternative proof of \eqref{theorem.radial.estimate.triple}.
Finally, we show in Part 4 that $p_2$ is the largest exponent for which \eqref{theorem.radial.estimate.triple} holds.

\medskip

\noindent{\it Part 1.}\quad{}We proceed now to show estimate \eqref{lemma.radial.estimate.Morrey} for $1\leq q<p_2$. All the constants $C$ will depend only on $n,p,\lambda$, and $q$.
On the one hand we have
\[
\begin{split}
 &\|u\|_{L^{q}(B_1(0))}^q
= C
\int_0^1
|u( r )|^q\,r^{n-1}
\,dr
\\
&\leq C
\int_0^1
\left(
\int_r^1
|u'(s)|
\,ds
\right)^q\,r^{n-1}
\,dr
= C
\sum_{j=1}^{\infty}
\int_{2^{-j}}^{2^{1-j}}
\left(
\int_r^1
|u'(s)|
\,ds
\right)^q\,r^{n-1}
\,dr
\\
&\leq C
\sum_{j=1}^{\infty}
\int_{2^{-j}}^{2^{1-j}}
\left(
\int_{2^{-j}}^1
|u'(s)|
\,ds
\right)^q
\,r^{n-1}
\,dr
= C\,\frac{2^n-1}{n}
\sum_{j=1}^{\infty}
2^{-jn}
\left(
\int_{2^{-j}}^1
|u'(s)|
\,ds
\right)^q.
\end{split}
\]
Now, H\"older's inequality yields
\[
\begin{split}
\int_{2^{-j}}^1
|u'(s)|
\,ds
&=
\sum_{i=1}^{j}
\int_{2^{-i}}^{2^{1-i}}
|u'(s)|
\,ds
\leq
\sum_{i=1}^{j}
2^{-i\frac{p-1}{p}}
\left(\int_{2^{-i}}^{2^{1-i}}
|u'(s)|^p
\,ds\right)^\frac{1}{p}
\\
&\leq
\sum_{i=1}^{j}
2^{-i\frac{p-1}{p}+i\frac{n-1}{p}}
\left(\int_{2^{-i}}^{2^{1-i}}
|u'(s)|^p\,s^{n-1}
\,ds\right)^\frac{1}{p}
\\
&\leq
C
\sum_{i=1}^{j}
2^{-i\frac{p-n}{p}}
\left(
\int_{B_{2^{1-i}}(0)}
|\nabla u(x)|^p
\,dx
\right)^\frac{1}{p}
\\
&=
C
\sum_{i=1}^{j}
2^{-i\frac{p-n}{p}+(1-i)\frac{n-\lambda}{p}}
\left(2^{(1-i)(\lambda-n)}
\int_{B_{2^{1-i}}(0)\cap B_{1}(0)}
|\nabla u(x)|^p
\,dx
\right)^\frac{1}{p}
\\
&\leq
C
\sum_{i=1}^{j}
2^{i\frac{\lambda-p}{p}}
\|
\nabla u
\|_{M^{p,\lambda}(B_{1}(0))}
\leq
C\,
2^{j\frac{\lambda-p}{p}}
\|
\nabla u
\|_{M^{p,\lambda}(B_{1}(0))}.
\end{split}
\]
Therefore, we obtain
\begin{equation*}
\begin{split}
 \|u\|_{L^{q}(B_1(0))}^q
\leq C
\sum_{j=1}^{\infty}
2^{j\frac{q(\lambda-p)-np}{p}}
\|\nabla u\|_{M^{p,\lambda} (B_1(0))}^q,
\end{split}
\end{equation*}
and the series is convergent since $q< p_2$.

\medskip

\noindent{\it Part 2.}\quad{}In order to show that $q< p_2$ is the optimal range of exponents in estimate \eqref{lemma.radial.estimate.Morrey}, let $q\geq p_2$ and consider the function
$u_\alpha(x)=|x|^{-\alpha}-1$, extended by zero outside $B_1(0)$,
with $\alpha=(\lambda-p)/p$.
 Notice that $u_\alpha$ vanishes on $\partial B_1(0)$ and that $\|u_\alpha\|_{L^{q}(B_1(0))}=\infty$,
since $n\leq\alpha q$.



To show that 
$\|\nabla u_\alpha\|_{M^{p,\lambda} (B_1(0))}$
is finite, let $y\in \overline{B}_1(0)$ and $r>0$. Observe that we can write $-y$ instead of $y$ in the definition of the Morrey norm. Then,
\[
\begin{split}
\int_{B_1(0)\cap B_r(-y)} |\nabla u_\alpha(x)|^{p}\, dx
&=
C
\int_{B_1(0)\cap B_r(-y)} |x|^{-\alpha p-p}\, dx
\\
&\leq
C
\int_{ B_r(-y)} |x|^{-\alpha p-p}\, dx
=
C
\int_{B_r(0)} |z-y|^{-\alpha p-p}\, dz
\end{split}
\]
by the  change of variables $z=x+y$.
Notice that, upon a rotation, we can assume $y=y_1e_1=(y_1,0,\ldots,0)$ with $y_1\geq0$. Denote
\[
J(y_1)=
\int_{B_r(0)} |z-y_1e_1|^{-\alpha p-p}\, dz.
\]
{We can now apply Lemma \ref{lemma.monotonicity} (with $\eta=0$ and $h\equiv1$) and conclude that $ J$ is non-increasing in $[0,\infty)$. Therefore, $J(y_1)\leq J(0)$ for all $y_1\geq0$.}

As a consequence, we have that 
\[
r^{\lambda-n}
\int_{B_1(0)\cap B_r(-y)} |\nabla u_\alpha(x)|^{p}\, dx
\leq
C
r^{\lambda-n}
\int_{0}^{r}
s^{n-\alpha p -p-1}\,ds
\leq
C
\]
independently of $r$, by our choice of $\alpha$.

\medskip

\noindent{\it Part 3a.}\quad{}We give here a first proof of estimate \eqref{theorem.radial.estimate.triple}. Recall that $u$ has compact support. 
Since $u$ is radial, it suffices to show that 
\begin{equation}\label{triple.norm.radial.proof.lemma}
\left(\int_0^\infty | u( r )|^\frac{np}{\lambda-p}\, r^{n-1}\, dr\right)^\frac{\lambda-p}{np}
\leq C\,
\left(\int_0^\infty | u'(r)|^{p} \,r^{\lambda-1}\, dr\right)^\frac{1}{p}.
\end{equation}

In fact, we are going to prove \eqref{triple.norm.radial.proof.lemma} with the best constant.
For this, we perform the change of variables $r=s^a$ with $a=p/(n-\lambda+p)$ and $v(s)=u(s^a)$ in the integrals on both sides of \eqref{triple.norm.radial.proof.lemma}. We get 
\[
\int_0^\infty | u( r )|^\frac{np}{\lambda-p}\, r^{n-1}\, dr
=
a
\int_0^\infty | v( s )|^\frac{np}{\lambda-p}\, s^{an-1}\, ds
\]
and
\[
\int_0^\infty | u'(r)|^{p} \,r^{\lambda-1}\, dr
=
a^{1-p}
\int_0^\infty | v'(s)|^{p} \,s^{an-1}\, ds.
\]
Observe that after the change of variables, both integrals are weighted by the same power $an-1\geq0$.  In this way, \eqref{triple.norm.radial.proof.lemma} becomes a Bliss type inequality that was proved by Talenti~\cite{Talenti.1976}. From \cite[Lemma 2]{Talenti.1976} (or also from \cite[Theorem 1.3]{Cabre.Ros-Oton} applied in dimension~$1$ with $D=A_1+1=an$) we obtain that
\[
\left(\int_0^\infty | v( s )|^\frac{np}{\lambda-p}\, s^{an-1}\, ds\right)^\frac{\lambda-p}{np}
\leq C_p\,
\left(\int_0^\infty | v'(s)|^{p} \,s^{an-1}\, ds\right)^\frac{1}{p},
\]
with an explicit value of its best constant $C_p$.
Inequality \eqref{triple.norm.radial.proof.lemma}, and thus \eqref{theorem.radial.estimate.triple}, is now established.

Moreover, it is also shown in \cite[Lemma 2]{Talenti.1976} (see also \cite{Cabre.Ros-Oton}) that
when $1 < p < \lambda$ the constant $C_p$ is attained in $W_0^{1,p}(\mathbb{R},|s|^{an-1} ds)$, the closure of ${C}_{c}^{1}(\mathbb{R})$ under the norm $(\int_{\mathbb{R}} (|u|^p+|\nabla u|^p) |s|^{an-1} ds)^{1/p}$,
 by the functions 
\[
v(s)= \left(c_1+c_2|s|^\frac{p}{p-1}\right)^\frac{p-\lambda}{n+p-\lambda},
\]
where $c_1,c_2>0$ are arbitrary constants.
On the other hand, when $p = 1$, the constant $C_1$ is not attained by any function in $W_0^{1,1}(\mathbb{R},|s|^{an-1} ds)$.
 Note however, that knowing the best constant for \eqref{triple.norm.radial.proof.lemma} does not  ensure that we know the best  constant for \eqref{theorem.radial.estimate.triple}. Indeed, for 
\[
u_{c_{1},c_{2}}(x)= \left(c_1+c_2|x|^\frac{n+p-\lambda}{p-1}\right)^\frac{p-\lambda}{n+p-\lambda},\qquad c_{1},c_{2}>0,
\]
we have seen that
\[
\|u_{c_{1},c_{2}}\|_{L^{p_2}(\mathbb{R}^n)}
=
C
\int_{\mathbb{R}^n} |\nabla u_{c_{1},c_{2}}(x)|^{p} \,|x|^{\lambda-n}\, dx
\]
with the best constant  $C$, but on the other hand we do not know if the supremum in the ``triple norm" is attained at $y=0$. That is, we do not know if
 \begin{equation}\label{what.we.want}
\int_{\mathbb{R}^n} |\nabla u_{c_{1},c_{2}}(x)|^{p} \,|x|^{\lambda-n}\, dx
=
 \opnorm{\nabla u_{c_{1},c_{2}}}_{p,\lambda;\mathbb{R}^n}.
 \end{equation}
The  point here is that $|\nabla u_{c_1,c_2}|^p$ is zero at the origin, instead of blowing up as in the other cases that we consider in the paper. Therefore,
 \[
\int_{\mathbb{R}^n} |\nabla u_{c_{1},c_{2}}(x)|^{p} \,|x-y|^{\lambda-n}\, dx
\]
could perhaps be increasing in $|y|$ in an interval near the origin, and then decrease to 0 as $|y|\to\infty$. In this setting, the reflection argument in the proof of Lemma \ref{lemma.monotonicity} does not work and we cannot conclude \eqref{what.we.want} as before.

\medskip

\noindent{\it Part 3b.}\quad{}We will provide here an alternative proof of estimate \eqref{theorem.radial.estimate.triple}. First, we establish the case $p=1$ by proving
\begin{equation}\label{estimate.intro.triplenorm.radial}
\left(\int_0^\infty | u( r )|^\frac{n}{\lambda-1}\, r^{n-1}\, dr\right)^\frac{\lambda-1}{n}
\leq C\,
\int_0^\infty | u'(r)| \,r^{\lambda-1}\, dr
\end{equation}
for every $\lambda\in(1,n)$. One can then 
 deduce the general case applying \eqref{estimate.intro.triplenorm.radial} with $(n(p-1)+\lambda)/p$ instead of $\lambda$ to the function $u^b$ with $b=1+n(p-1)/(\lambda-p)$. Note that $1< (n(p-1)+\lambda)/p <n$.

To show \eqref{estimate.intro.triplenorm.radial}, notice that we can assume $u\geq0$. Furthermore, we can also assume that $u$ is radially decreasing. This follows from \cite[Proposition 4.2]{Cabre.Ros-Oton} (see also \cite{Talenti}) applied to inequality \eqref{estimate.intro.triplenorm.radial} after changing variables as in Part 3 in order to guarantee that both sides of the inequality are weighted by the same power. Now, on the one hand, the change of variables $t=u(r)$ yields
\begin{equation}\label{part1.radial.eq.phi}
\int_0^\infty | u'(r)| \,r^{\lambda-1}\, dr
=
-
\int_0^\infty u'(r) \,r^{\lambda-1}\, dr
=
\int_0^{\max(u)}
\varphi(t)^\frac{\lambda-1}{n}
\,dt
\end{equation}
for $\varphi(t)=\big|\{r:\ u( r)>t\}\big|^n$.
 On the other hand, by Cavalieri's principle
\begin{equation}\label{part2.radial.eq.phi}
\begin{split}
 \int_0^\infty u( r )^\frac{n}{\lambda-1}\, r^{n-1}\, dr
 =
 \frac{1}{n}
 \int_0^\infty
 \frac{n}{\lambda-1}\,
 t^{\frac{n}{\lambda-1}-1} \varphi(t)
 \,dt.
\end{split}
\end{equation}

We conclude by proving that 
\begin{equation}\label{lemma.decreasing.functions}
\int_0^\infty
q\,t^{q-1}\,\varphi(t)
\,dt
\leq
\left(
\int_0^\infty
\varphi(t)^\frac{1}{q}
\,dt
\right)^q
\end{equation}
for every non-increasing function $\varphi=\varphi(t)$ and every $q>1$. 
Consequently, \eqref{estimate.intro.triplenorm.radial} will follow from \eqref{part1.radial.eq.phi}, \eqref{part2.radial.eq.phi}, and \eqref{lemma.decreasing.functions}.
To prove inequality \eqref{lemma.decreasing.functions}, denote
\[
F_1(s)=\int_0^s
q\,t^{q-1}\,\varphi(t)
\,dt
\qquad\textnormal{and}\qquad
F_2(s)=
\left(
\int_0^s
\varphi(t)^\frac{1}{q}
\,dt
\right)^q,
\]
and notice that $F_1(0)=F_2(0)$, while
\[
\begin{split}
F_2'(s)
&=q\varphi(s)^\frac{1}{q}\left(
\int_0^s
\varphi(t)^\frac{1}{q}
\,dt
\right)^{q-1}\\
&\geq
q\varphi(s)^\frac{1}{q}\left(
s\varphi(s)^\frac{1}{q}
\right)^{q-1}
=
q\,s^{q-1}\,\varphi(s)
=
F_1'(s).
\end{split}
\]
This establishes \eqref{lemma.decreasing.functions} and hence concludes the proof of \eqref{theorem.radial.estimate.triple}.

\medskip

\noindent{\it Part 4.}\quad{}Finally, we show that $p_2$ is the largest exponent for which \eqref{theorem.radial.estimate.triple} holds. Let $q>p_2$ and consider the function
$u_\alpha(x)=\left(|x|^{-\alpha}-1\right)_+$
with $n/q\leq\alpha<(\lambda-p)/p$. Notice that 
\[
\|u_\alpha\|_{L^{q}(\mathbb{R}^{n})}^q
=
\int_{B_1(0)}
\big(|x|^{-\alpha}-1\big)^q
\,dx
=
C
\int_{0}^{1}
(1-r^{\alpha})^q\, r^{n-\alpha q-1}
\,dr
=\infty,
\]
by our choice of $\alpha$.

 On the other hand, we are going to see that
$\opnorm{\nabla u_\alpha}_{p,\lambda;\mathbb{R}^{n}}<\infty$, thus contradicting the inequality. To prove this, we claim that $y=0$ realizes the supremum in the definition of
$\opnorm{\nabla u_\alpha}_{p,\lambda;\mathbb{R}^{n}}$, and hence
\[
\opnorm{\nabla u_\alpha}_{p,\lambda;\mathbb{R}^{n}}^{p} 
= 
\int_{B_1(0)} |\nabla u_{\alpha}(x)|^{p}\,|x|^{\lambda-n}\, dx
=
C
\int_{0}^{1}
 r^{\lambda-\alpha p-p-1}
\,dr
<\infty,
\]
since $\lambda-\alpha p-p >0$.

We conclude proving the claim by monotonicity. For $y\in\mathbb{R}^n$, we have that
\[
\int_{B_1(0)} |\nabla u_{\alpha}(x)|^{p}\,|x-y|^{\lambda-n}\, dx
=
\alpha^{p}
\int_{B_1(0)} |x|^{-\alpha p-p}\,|x-y|^{\lambda-n}\, dx.
\]
Notice that, upon a rotation, we can assume $y=y_1e_1=(y_1,0,\ldots,0)$.
Denote
\[
J(y_1)=
\int_{B_1(0)} |x|^{-\alpha p-p}\,|x-y_1e_1|^{\lambda-n}\, dx.
\]
Since the function $J$ is under the hypotheses of Lemma \ref{lemma.monotonicity} (with $\eta=0$), we conclude that
 $ J$ is non-increasing in $[0,\infty)$, and therefore that $J(y_1)\leq J(0)$ for all $y_1\geq0$.
\end{proof}


\section{Embeddings in the general case: Proof of Theorem \ref{Adams.original.result}}\label{section.embeddings.general}

For the reader's convenience, we provide in this section D. R. Adams' proof of Theorem~\ref{Adams.original.result}, see \cite[Theorem 3.1]{Adams.1975}.
Recall that we consider a Lipschitz function $u:\mathbb{R}^n\to\mathbb{R}$ with $u\equiv0$ in $\mathbb{R}^n\setminus B_1(0)$, so that integrals in $\mathbb{R}^n$ and integrals in $B_1(0)$ coincide.

\begin{proof}[Proof of Theorem \ref{Adams.original.result}]
The proof is based on the following two claims,
\begin{equation}\label{estimate.pointwise.u.Riesz.potential}
|u(x)|\leq C (I_1|\nabla u|)(x),\qquad\textrm{a.e.}\ x.
\end{equation}
and 
\begin{equation}\label{estimate.adams.fact2}
\big|I_1 |\nabla u|\big|(x)\leq C\big(M_{\lambda/p}|\nabla u|(x)\big)^\frac{p}{\lambda}\big(M_{0}|\nabla u|(x)\big)^{1- \frac{p}{\lambda}},
\end{equation}
where 
\[
I_1f(x):=\int_{\mathbb{R}^n} f(y)\,|y-x|^{1-n}\, dy
\]
 is the Riesz potential of $f$, and 
\[
M_{\beta}f(x)
:=
\sup_{r>0}\left( r^{\beta-n} \int_{B_r(x)} |f(z)|\, dz\right),\qquad 0\leq\beta\leq n
\]
is the maximal function with parameter $\beta$. 

Once \eqref{estimate.pointwise.u.Riesz.potential} and \eqref{estimate.adams.fact2} are established, we can finish the proof as follows. By H\"older's inequality, we have
\[
\begin{split}
M_{\lambda/p}|\nabla u|(x)
&=
\sup_{r>0} \left(r^{\frac{\lambda}{p}-n} \int_{ B_r(x)} |\nabla u(z)|\, dz\right)
\\
&\leq
C
\sup_{r>0} \left(r^{\lambda-n}\int_{B_1(0)\cap B_r(x)} |\nabla u(z)|^{p}\, dz\right)^\frac{1}{p}
\leq
C\|\nabla u\|_{M^{p,\lambda}(B_1(0))}.
\end{split}
\]
(recall that $u\equiv0$ in $\mathbb{R}^n\setminus B_1(0)$).
 Then, \eqref{estimate.pointwise.u.Riesz.potential} and \eqref{estimate.adams.fact2} give
\[
|u(x)|
\leq
 C\big(M_{\lambda/p}|\nabla u|(x)\big)^\frac{p}{\lambda}\big(M_{0}|\nabla u|(x)\big)^{1-\frac{p}{\lambda}}
\leq
 C\|\nabla u\|_{M^{p,\lambda}(B_1(0))}^\frac{p}{\lambda}\big(M_{0}|\nabla u|(x)\big)^{1- \frac{p}{\lambda}}
\]
almost everywhere. It follows that 
\[
\begin{split}
\|u\|_{L^{q}(B_1(0))}^q
&\leq 
 C^q\,\|\nabla u\|_{M^{p,\lambda}(B_1(0))}^\frac{p q}{\lambda}
 \int_{B_1(0)}
 \big(M_{0}|\nabla u|(x)\big)^{ \frac{p q}{p_1}}
 \,dx
 \\
& \leq 
 C^q\,\|\nabla u\|_{M^{p,\lambda}(B_1(0))}^\frac{p q}{\lambda}
\,
 |B_1(0)|^{1-\frac{q}{p_1}}
 \,
 \|M_{0}|\nabla u|\|_{L^{p}(B_1(0))}^\frac{p q}{p_1},
 \end{split}
\]
where we have applied H\"older's inequality with exponents $p_1/(p_1-q)$ and $p_1/q$.

By the well-known $L^p$ estimate for the maximal function $M_0$
when $p>1$, there exists a constant $C$ depending only on $n$ and $p$ such that
\[
\|M_{0}|\nabla u|\|_{L^{p}(B_1(0))}
\leq
C
\|\nabla u\|_{L^{p}(B_1(0))},
\]
and therefore
\[
\|u\|_{L^{q}(B_1(0))}
 \leq 
 C
\,
 |B_1(0)|^{\frac1q-\frac{1}{p_1}}
 \,
 \|\nabla u\|_{M^{p,\lambda}(B_1(0))}^{\frac{p}{\lambda}+\frac{p}{p_1}}
 \leq
 C\,\|\nabla u\|_{M^{p,\lambda}(B_1(0))}
 \]
with $C$ depending only on $n, p,$ and $\lambda$ as desired.

Therefore, it remains to prove claims \eqref{estimate.pointwise.u.Riesz.potential} and \eqref{estimate.adams.fact2}.

Consider first estimate \eqref{estimate.pointwise.u.Riesz.potential}.
To prove it, notice that for $\sigma\in\mathbb{R}^n$ with $|\sigma|=1$ 
\[
|u(x)|=
\left|
-\int_0^\infty
\frac{d}{dr} u(x+r\sigma)
\,dr
\right|
\leq
\int_0^\infty
|\nabla u(x+r\sigma)|
\,dr.
\]
Then, integrating on $\sigma$ we get
\[
|u(x)|
\leq
C \int_0^\infty\int_{\partial B_1(0)}
\big(|\nabla u(x+r\sigma)| r^{1-n}\big)
r^{n-1}\,d\sigma dr
=
C\,(I_1|\nabla u|)(x)
\]
and \eqref{estimate.pointwise.u.Riesz.potential} is proved.

Next, consider estimate \eqref{estimate.adams.fact2}. We reproduce the argument in \cite{Adams.1975} to show that for a given function $f$ with compact support in $\mathbb{R}^n$, we have
 \begin{equation}\label{Inequality.Adams}
\big|I_1 f(x)\big|\leq C\big(M_{\lambda/p} f(x)\big)^\frac{p}{\lambda}\big(M_{0} f(x)\big)^{1- \frac{p}{\lambda}},
\end{equation}
where $C$ depends only on $n,p,$ and $\lambda$.

 For $f\not\equiv0$, let $\delta>0$ to be determined later and
 set 
 \[
 \begin{split}
 I_1 f(x)&=\int_{\mathbb{R}^n} f(y)\,|y-x|^{1-n}\, dy\\
 &=\int_{\{y:\ |x-y|<\delta\}}f(y)\,|y-x|^{1-n}\, dy
 +
 \int_{\{y:\ |x-y|\geq\delta\}}
 f(y)\,|y-x|^{1-n}\, dy
 \\
 &=I+I'.
 \end{split}
 \]
 Let 
 \[
 a_{k}(x)=\{y:\ 2^k\delta\leq|x-y|<2^{k+1}\delta\}\qquad\textrm{for}\ k\in\mathbb{Z}.
 \]
 Then,
\[
\begin{split}
 |I|&\leq\sum_{k=1}^\infty\int_{a_{-k}(x)}|f(y)|\,|x-y|^{1-n}\,dy
 \\
 &\leq\sum_{k=1}^\infty\big(2^{-k}\delta\big)^{1-n}\big(2^{-k+1}\delta\big)^{n} M_{0}f(x)
 = 2^n\delta M_0f(x).
\end{split}
\]
 Similarly,
\[
\begin{split}
 |I'|&\leq\sum_{k=0}^\infty\int_{a_{k}(x)}|f(y)|\,|x-y|^{1-n}\,dy
 \\
 &\leq\sum_{k=0}^\infty\big(2^{k}\delta\big)^{1-n}\big(2^{k+1}\delta\big)^{n-
 \frac{\lambda}{p}} M_{\lambda/p}f(x)
 =C\delta^{1-\frac{\lambda}{p}} M_{\lambda/p}f(x),
\end{split}
\]
since $p<\lambda$. The choice 
\[
\delta=\delta(x)=\left(
\frac{M_{\lambda/p}f(x)}{M_{0}f(x)}\right)^\frac{p}{\lambda}
\]
finally gives
\eqref{Inequality.Adams}.
\end{proof}

\section{Proof of Theorem \ref{intro.counterexamples} in the case when $\lambda$ is an integer}\label{section.lambda.integer}

In this section we prove Theorem \ref{intro.counterexamples} when $\lambda$ is an integer. The argument is very simple. The  case $\lambda\notin\mathbb{Z}$ is the core of our paper and will be considered in Section \ref{section.lambda.noninteger}. 

As mentioned in the introduction, the choice of the counterexample when $\lambda\in\mathbb{Z}$ was hinted by the number $p_1=\lambda p/(\lambda-p)$, which can be thought of as the Sobolev exponent in dimension $\lambda$. Then, it is natural to choose a function that provides a counterexample for the Sobolev inequality in dimension $\lambda$ and then look at this function embedded in the $n$-dimensional space. Namely, we take
\begin{equation}\label{eq.u.alpha.integer.case}
u(x)=\left(|x'|^{-\alpha}-2^\alpha\right)_{+}\,\xi(|x''|) ,
\end{equation}
where $x=(x',x'')\in\mathbb{R}^{\lambda}\times\mathbb{R}^{n-\lambda}$ 
and $\xi:\mathbb{R}^+\to[0,1]$ is a cutoff function with $\xi\equiv 1$ in $[0,1/2)$ and $\xi\equiv0$ in $\mathbb{R}^+\setminus[0,\sqrt{3}/3).$ Note that clearly $u$ has support in $B_1(0)\subset\mathbb{R}^n$.

The rest of the section is devoted to show the following result, which proves Theorem~\ref{intro.counterexamples} when $\lambda$ is an integer.
\begin{proposition}\label{prop.integer.case}
Let $\lambda$ be an integer such that $1<p<\lambda<n$ and assume that $q> p_1:=\lambda  p/(\lambda-p)$. Then, for $u$ given by \eqref{eq.u.alpha.integer.case}, we have that $\|u\|_{L^{q}(B_1)}=\infty$ and $\opnorm{\nabla u}_{p,\lambda;B_1}<\infty$~if
\[
\frac{\lambda}{q}\leq\alpha<\frac{\lambda-p}{p}.
\]
This proves the optimality of the range $q\leq p_1$ for \eqref{estimate.intro.triple} when $\lambda\in\mathbb{Z}$,
and in turn also for \eqref{estimate.intro.morrey}. 
\end{proposition}

\begin{proof}
For every $y\in \overline{B}_1^{(n)}$ we have that
\[
\begin{split}
\int_{B_1^{(n)}} &|\nabla u(x)|^{p}\,|x-y|^{\lambda-n}\, dx\\
&\leq
C
\int_{B_{1}^{(\lambda)}}\int_{B_{1}^{(n-\lambda)}} |x'|^{-\alpha p-p}\,|x-y|^{\lambda-n}\, dx''dx'
\\
&= 
C
\int_{B_{1}^{(\lambda)}} |x'|^{-\alpha p-p}\int_{B_{1}^{(n-\lambda)}}\left(1+\left(\frac{|x''-y''|}{|x'-y'|}\right)^2\right)^{\frac{\lambda-n}{2}}|x'-y'|^{\lambda-n}\, dx''dx',
\end{split}
\]
for some constant $C$ independent of $y$.
The change of variables $z=\frac{x''-y''}{|x'-y'|}$ yields
\[
\begin{split}
\int_{B_{1}^{(n-\lambda)}}&\left(1+\left(\frac{|x''-y''|}{|x'-y'|}\right)^2\right)^{\frac{\lambda-n}{2}}|x'-y'|^{\lambda-n}\,dx''\\
&\leq
\int_{B_{2/|x'-y'|}^{(n-\lambda)}}\left(1+|z|^2\right)^{\frac{\lambda-n}{2}}\,dz
=
C
\int_{0}^{\frac{2}{|x'-y'|}}\left(1+r^2\right)^{\frac{\lambda-n}{2}}r^{n-\lambda-1}\,dr\\
&\leq
C
\int_{0}^{\frac{2}{|x'-y'|}}\max\{1,r\}^{\lambda-n}r^{n-\lambda-1}\,dr
\leq C\left(1+\left|\log|x'-y'|\right|\right).
\end{split}
\]
Therefore, we have
\[
\int_{B_1^{(n)}} |\nabla u(x)|^{p}\,|x-y|^{\lambda-n}\, dx
\leq
C
\int_{B_{1}^{(\lambda)}} |x'|^{-\alpha p-p}
\left(1+\left|\log|x'-y'|\right|\right)\, dx'.
\]
We claim that the last integral is bounded uniformly in $y'\in \overline{B}_1^{(\lambda)}$. 
To verify this, since $\left|\log|x'-y'|\right| \leq \log 2$ for $x'\in B_1^ {(\lambda)}\setminus B_1(y')$ and $\lambda-\alpha p-p >0$,
it suffices to control the integral over $B_1(y')$. But then, calling $z:=y'-x'$, the integral becomes
$\int_{B_{1}^{(\lambda)}} h(z)|z-y'|^{-\alpha p-p}\, dz$ with $h(z)=1+\left| \log |z|\right|=1-\log |z|$ for $z\in B_{1}^{(\lambda)}$.
Now, since $h$ is non-negative and radially decreasing in $B_{1}^{(\lambda)}$, we can apply Lemma~\ref{lemma.monotonicity}
with $\eta=0$ and conclude that the largest value of the integral corresponds to $y'=0$. But since we have assumed $\lambda-\alpha p-p >0$,
the integral with $y'=0$ is finite.

On the other hand,
\[
\begin{split}
 \|u\|_{L^{q}(B_1)}^{q}
&\geq
\int_{B_{1/2}^{(\lambda)}}\int_{B_{1/2}^{(n-\lambda)}} u^q\,dx''dx' \\
&=
C
\int_{B_{1/2}^{(\lambda)}}
\left(|x'|^{-\alpha}-2^\alpha\right)^{q}
\,dx'
=
C
\int_{0}^{\frac12}
\left(r^{-\alpha}-2^\alpha\right)^{q}r^{\lambda-1}
\,dr
\end{split}
\]
and the last integral is divergent since $\lambda\leq\alpha q$ by hypothesis.
\end{proof}

\section{Proof of Theorem \ref{intro.counterexamples} in the general case}\label{section.lambda.noninteger}

In this section we conclude the proof of Theorem \ref{intro.counterexamples} by considering the case when $\lambda$ is not an integer.

Let us motivate first the case $n-1<\lambda<n$. As we have seen in Section \ref{section.lambda.integer}, when $\lambda$ is equal to $n-1$ expression
\eqref{eq.u.alpha.integer.case} provides a counterexample to the embedding \eqref{estimate.intro.triple}
when $q>p_1$, and therefore also to \eqref{estimate.intro.morrey} in view of \eqref{relation.norms}.
On the other hand, when $\lambda$ is equal to $n$ the Morrey and 
triple norms
coincide with the Sobolev norm and 
$u(x)=\left(|x|^{-\alpha}-2^\alpha\right)_{+}$ provides a counterexample to embedding \eqref{estimate.intro.triple}
for $q>p_1=p^*$ (in this case \eqref{estimate.intro.triple} is simply the Sobolev embedding).
In both cases the function that yields the counterexample is basically a negative power of the distance function, either to the origin in the case $\lambda=n$, or to a line when $\lambda=n-1$.
Therefore, when $\lambda$ is strictly between $n-1$ and $n$,
 a negative power of the distance to a fractal set of non-integer dimension $n-\lambda$ is a natural candidate
to be a counterexample to inequality \eqref{estimate.intro.triple}.

Let us describe precisely the functions that provide the counterexample.
When $n-1<\lambda<n$, we consider  
\begin{equation}\label{function.counterexample.in.dimension.between.n-1.n}
u_{\alpha,n}(x)=\left(\textrm{dist}(x,\mathcal{C}_{n,\lambda})^{-\alpha}-4^\alpha\right)_{+}
\end{equation}
for $\mathcal{C}_{n,\lambda}=\{0\}\times C_{\gamma}\subset\mathbb{R}^{n-1}\times[-1/2,1/2]$, where $C_{\gamma}$ is a generalized Cantor set with parameter 
\begin{equation}\label{defn.gamma.sec.5}
\gamma=1-2^{1-\frac{1}{n-\lambda}}\in(0,1).
\end{equation}

The generalized Cantor set $C_{\gamma}$ (see \cite{Falconer})  is obtained from the interval $[-1/2,1/2]$ by removing at iteration $j=1,2,\ldots$ the central interval of length $\gamma\,l_{j-1}$ from each remaining segment of length $l_{j-1}=\big((1-\gamma)/2\big)^{j-1}$; 
see Figure \ref{figure.gencantor}. A precise expression for $C_\gamma$ is given later in \eqref{eq.formula.complement.cantor}}, \eqref{GLM}, and \eqref{eq.definition.h_{l,m}}.
The usual Cantor set corresponds to $\gamma=1/3$.

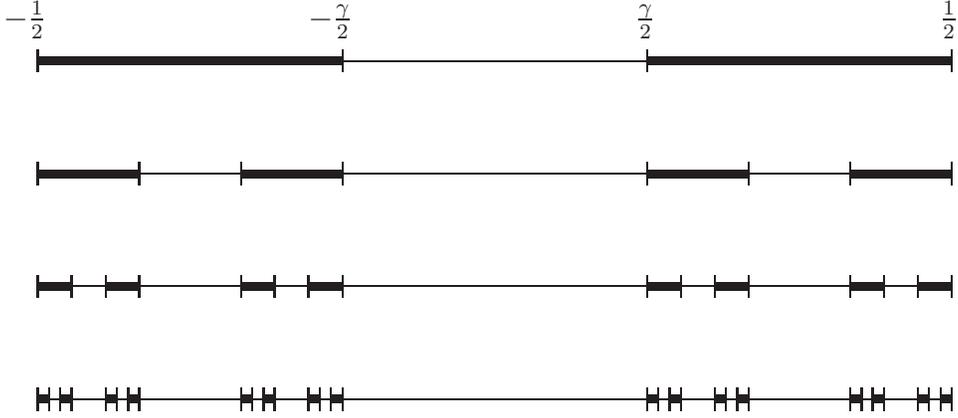
\begin{figure}
\setlength{\unitlength}{.15cm}
\begin{picture}(80,40)(0,-3)

 \put(0,0){\line(1,0){81}}
 \put(0,10){\line(1,0){81}}
 \put(0,20){\line(1,0){81}}
 \put(0,30){\line(1,0){81}}
 
 \put(-3,33){$-\frac12$}
 \put(80,33){$\frac12$} 
 
 \put(24,33){$-\frac{\gamma}{2}$}
 \put(53,33){$\frac{\gamma}{2}$}
 
 

 \multiput(0,29)(27,0){4}{\line(0,1){2}} 

 \multiput(0,19)(9,0){4}{\line(0,1){2}} 
 \multiput(54,19)(9,0){4}{\line(0,1){2}} 

 \multiput(0,9)(3,0){4}{\line(0,1){2}} 
 \multiput(18,9)(3,0){4}{\line(0,1){2}} 
 \multiput(54,9)(3,0){4}{\line(0,1){2}} 
 \multiput(72,9)(3,0){4}{\line(0,1){2}} 
 
 \multiput(0,-1)(1,0){4}{\line(0,1){2}} 
 \multiput(6,-1)(1,0){4}{\line(0,1){2}} 
 \multiput(18,-1)(1,0){4}{\line(0,1){2}} 
 \multiput(24,-1)(1,0){4}{\line(0,1){2}} 
 \multiput(54,-1)(1,0){4}{\line(0,1){2}} 
 \multiput(60,-1)(1,0){4}{\line(0,1){2}} 
 \multiput(72,-1)(1,0){4}{\line(0,1){2}} 
 \multiput(78,-1)(1,0){4}{\line(0,1){2}} 

 \linethickness{3pt}
 \multiput(0,30)(54,0){2}{\line(1,0){27}} 
 
 \multiput(0,20)(18,0){2}{\line(1,0){9}} 
 \multiput(54,20)(18,0){2}{\line(1,0){9}} 

 \multiput(0,10)(6,0){2}{\line(1,0){3}} 
 \multiput(18,10)(6,0){2}{\line(1,0){3}} 
 \multiput(54,10)(6,0){2}{\line(1,0){3}} 
 \multiput(72,10)(6,0){2}{\line(1,0){3}} 
 
 \multiput(0,0)(2,0){2}{\line(1,0){1}} 
 \multiput(6,0)(2,0){2}{\line(1,0){1}} 
 \multiput(18,0)(2,0){2}{\line(1,0){1}} 
 \multiput(24,0)(2,0){2}{\line(1,0){1}} 
 \multiput(54,0)(2,0){2}{\line(1,0){1}} 
 \multiput(60,0)(2,0){2}{\line(1,0){1}} 
 \multiput(72,0)(2,0){2}{\line(1,0){1}} 
 \multiput(78,0)(2,0){2}{\line(1,0){1}} 

\end{picture}
\caption{Construction of the generalized Cantor set $C_{\gamma}$}
\label{figure.gencantor}
\end{figure}

The reason for our choice of $\gamma$ in \eqref{defn.gamma.sec.5} is that
the Hausdorff dimension of $C_{\gamma}$ is 
\[
\frac{-\log{2}}{\log{\frac{1-\gamma}{2}}}=n-\lambda
\]
 (see \cite[Theorem 9.3]{Falconer}). Thus, letting $\lambda$ vary between $n-1$ and $n$ yields any fractal dimension between 0 and 1. In particular, \eqref{function.counterexample.in.dimension.between.n-1.n} somehow interpolates the integer cases $\lambda=n-1$ and $\lambda=n$.

Note that $u_{\alpha,n}$ has support in $B_{3/4}^{(n)}(0)$.     Indeed, if $y\in\mathcal{C}_{n,\lambda}$ then $|y|\leq1/2$, and thus $|x-y|\geq1/4$ if $|x|\geq3/4$; in particular $\textrm{dist}(x,\mathcal{C}_{n,\lambda})\geq1/4$ and $u_{\alpha,n}(x)=0$.

In the case when $k-1<\lambda< k$ for some integer $k\in\{2,\ldots, n-1\}$ we embed
 into $\mathbb{R}^n$ the counterexample in $\mathbb{R}^k$ by means of an appropriate cutoff function (as we did in the previous section when $\lambda\in\mathbb{Z}$). In this way, we reduce the proof to the case $n-1<\lambda<n$.
More precisely, we consider
\begin{equation}\label{counterexample.noninteger.general}
u_\alpha(x)=u_{\alpha,k}(x')\,\xi(|x''|)\qquad\textrm{if}\ k<n,
\end{equation}
where $x=(x',x'')\in\mathbb{R}^{k}\times\mathbb{R}^{n-k}$, $u_{\alpha,k}$ is given by \eqref{function.counterexample.in.dimension.between.n-1.n} with $n$ replaced by $k$, i.e.,
\begin{equation}\label{counterexample.noninteger.general.2}
u_{\alpha,k}(x')=\left(\textrm{dist}(x',\mathcal{C}_{k,\lambda})^{-\alpha}-4^\alpha\right)_{+},
\end{equation}
and
 $\xi:\mathbb{R}^+\to[0,1]$ is a cutoff function with $\xi\equiv 1$ in $[0,1/2)$ and $\xi\equiv0$ in $\mathbb{R}^+\setminus[0,\sqrt{3}/3).$ Note that if $u_\alpha(x)\neq0$ then necessarily $|x'|\leq3/4$ and $|x''|\leq\sqrt{3}/3$; thus $|x|<1$.

The rest of the section is devoted to proving the following result, which together with Proposition \ref{prop.integer.case} completes the proof of Theorem \ref{intro.counterexamples}.
\begin{theorem}\label{general.result}
Let $p,\lambda,q\in\mathbb{R}$ be such that $1<p<\lambda<n$, $k-1<\lambda<k$ for some integer $k\in\{2,\ldots,n\}$, and $q\geq1$.
Consider
\[
u_\alpha(x) =
\left\{
\begin{split}
&\left(\textnormal{dist}(x,\mathcal{C}_{n,\lambda})^{-\alpha}-4^\alpha\right)_{+}\hspace{67pt}\textnormal{if}\ k=n \\
&\left(\textnormal{dist}(x',\mathcal{C}_{k,\lambda})^{-\alpha}-4^\alpha\right)_{+}\,\xi(|x''|) \hspace{29pt}\textnormal{if}\ k\in\{2,\ldots,n-1\},
\end{split}
\right.
\]
with $x=(x',x'')\in\mathbb{R}^{k}\times\mathbb{R}^{n-k}$, $\xi:\mathbb{R}^+\to[0,1]$ a cutoff function as described after \eqref{counterexample.noninteger.general.2}, and $\mathcal{C}_{k,\lambda}$ a set of Hausdorff dimension $k-\lambda$ given by
\[
\mathcal{C}_{k,\lambda}=\{0\}\times C_{\gamma}\subset\mathbb{R}^{k-1}\times[-1/2,1/2],
\]
where $C_{\gamma}\subset[-1/2,1/2]$ is a generalized Cantor set with parameter $\gamma=1-2^{1-\frac{1}{k-\lambda}}$.
Then, 
\[
\|u_\alpha\|_{L^{q}(B_1^{(n)})}=\infty\quad\textrm{and}\quad \opnorm{\nabla u_\alpha}_{p,\lambda;B_1^{(n)}}<\infty
\]
if
\begin{equation}\label{conditions.alpha.sect5}
\frac{\lambda}{q}\leq\alpha<\frac{\lambda-p}{p} .
\end{equation}
This proves the optimality of the range $q\leq p_1=\lambda p/(\lambda-p)$ for inequality \eqref{estimate.intro.triple}, and in turn also for \eqref{estimate.intro.morrey}. 
\end{theorem}

The proof of Theorem \ref{general.result} is divided into two parts. In the first part we reduce the computations from dimension $n$ to dimension $k$ with a similar argument to the one in the proof of Proposition \ref{prop.integer.case}. 
More precisely we will show, for $u_{\alpha,k}$ given by \eqref{counterexample.noninteger.general.2}, that
\[
\opnorm{\nabla u_\alpha}_{p,\lambda;B_1^{(n)}}\leq C\,\opnorm{\nabla u_{\alpha,k}}_{p,\lambda;B_1^{(k)}}
\]
and
\[
\|u_\alpha\|_{L^{q}(B_1^{(n)})}\geq C^{-1}\|u_{\alpha,k}\|_{L^{q}(B_{1}^{(k)})},
\]
for some constant $C$,
and hence that it is enough to study the case $k=n$ taking $u_\alpha=u_{\alpha,n}$. 
In the second part of the proof we will show that \eqref{conditions.alpha.sect5} leads to 
\[
\|u_{\alpha,n}\|_{L^{q}(B_{1}^{(n)})}
=
\infty\quad\textrm{and}\quad \opnorm{\nabla u_{\alpha,n}}_{p,\lambda;B_1^{(n)}}<\infty,
\]
as desired; this part is the content of the following two propositions.
\begin{proposition}\label{prop.Lp.norm.infinite}
Let $\lambda,q\in\mathbb{R}$ be such that $\lambda>n-1$ and
$q\geq1$. Consider
 $u_{\alpha,n}$ given by \eqref{function.counterexample.in.dimension.between.n-1.n}
with $\alpha$ such that $\lambda\leq\alpha q$. 
Then,
$\| u_{\alpha,n}\|_{L^{q}(B_1^{(n)})}=\infty$.
\end{proposition}

\begin{proposition}\label{prop.triple.norm.bounded}
Let $\lambda,p\in\mathbb{R}$ be such that 
 $n-1<\lambda<n$, $p>1$ and consider $u_{\alpha,n}$ given by \eqref{function.counterexample.in.dimension.between.n-1.n}
with $\alpha>0$ satisfying 
\begin{equation}\label{conditions.alpha}
\alpha<\frac{\lambda-p}{p}.
\end{equation}
Then,
\[
\opnorm{\nabla u_{\alpha,n}}_{p,\lambda;B_1^{(n)}}^{p} :=\sup_{y\in \overline{B}_1^{(n)} }\int_{B_1^{(n)}} |\nabla u_{\alpha,n}(x)|^{p}\,|x-y|^{\lambda-n}\, dx<\infty.
\]
\end{proposition}

Let us now prove Theorem \ref{general.result} assuming Propositions \ref{prop.Lp.norm.infinite} and \ref{prop.triple.norm.bounded}, which will be established afterwards.

\begin{proof}[Proof of Theorem \ref{general.result}.]
Since Propositions \ref{prop.Lp.norm.infinite} and \ref{prop.triple.norm.bounded} yield the result when $k=n$, let us assume $k<n$.
Let $y\in \overline{B}_1^{(n)}$.
By \eqref{counterexample.noninteger.general} and \eqref{counterexample.noninteger.general.2} we have 
\[
|\nabla u_\alpha(x)|
\leq
C
\big(
|\nabla u_{\alpha,k}(x')|
+
u_{\alpha,k}(x')
\big)
\leq
C
|\nabla u_{\alpha,k}(x')|
\]
for almost every $x\in B_1^{(n)}$, where we have used that the modulus of the gradient of a distance function is equal to $1$ a.e.\ Therefore,
\[
\begin{split}
&\int_{B_1^{(n)}} |\nabla u_\alpha(x)|^{p}\,|x-y|^{\lambda-n}\, dx
\leq
C
\int_{B_{1}^{(k)}}\int_{B_{1}^{(n-k)}} |\nabla u_{\alpha,k}(x')|^{p}\,|x-y|^{\lambda-n}\, dx''dx'
\\
&= 
C
\int_{B_{1}^{(k)}} |\nabla u_{\alpha,k}(x')|^{p}\,|x'-y'|^{\lambda-k}
\int_{B_{1}^{(n-k)}}\left(1+\left(\frac{|x''-y''|}{|x'-y'|}\right)^2\right)^{\frac{\lambda-n}{2}}|x'-y'|^{k-n}\, dx''dx'.
\end{split}
\]
The change of variables $z=\frac{x''-y''}{|x'-y'|}$ yields
\[
\begin{split}
\int_{B_{1}^{(n-k)}}&\left(1+\left(\frac{|x''-y''|}{|x'-y'|}\right)^2\right)^{\frac{\lambda-n}{2}}|x'-y'|^{k-n}\,dx''
\leq
\int_{B_{2/|x'-y'|}^{(n-k)}}\left(1+|z|^2\right)^{\frac{\lambda-n}{2}}\,dz
\\
&
=
C
\int_{0}^{\frac{2}{|x'-y'|}}\left(1+r^2\right)^{\frac{\lambda-n}{2}}r^{n-k-1}\,dr
\leq
C
\int_{0}^{\frac{2}{|x'-y'|}}\max\{1,r\}^{\lambda-n}\,r^{n-k-1}\,dr
\\
&
=
C
\left(
\int_{0}^{1}r^{n-k-1}\,dr
+
\int_{1}^{\frac{2}{|x'-y'|}}r^{\lambda-k-1}\,dr
\right)
\leq C
\end{split}
\]
independently of $y$.
Therefore, 
\[
\begin{split}
\int_{B_1^{(n)}} &|\nabla u_\alpha(x)|^{p}\,|x-y|^{\lambda-n}\, dx \\
&\leq
C \sup_{y'\in \overline{B}_{1}^{(k)}}
\int_{B_{1}^{(k)}} |\nabla u_{\alpha,k}(x')|^{p}\,|x'-y'|^{\lambda-k}\, dx'= C \,\opnorm{\nabla u_{\alpha,k}}_{p,\lambda;B_1^{(k)}},
\end{split}
\]
and Proposition \ref{prop.triple.norm.bounded} applied in $\mathbb{R}^k$, i.e., with $n$ replaced by $k$,
yields $\opnorm{\nabla u_\alpha}_{p,\lambda;B_1^{(n)}}<\infty$.

On the other hand,
\[
\|u_\alpha\|_{L^{q}(B_1^{(n)})}^{q}
=
\int_{B_{3/4}^{(k)}}u_{\alpha,k}^q(x')\int_{B_{\sqrt{3}/3}^{(n-k)}} \xi^q(|x''|) \,dx''dx' =
C \, \|u_{\alpha,k}\|_{L^{q}(B_{3/4}^{(k)})}^q,
\]
which is infinite by Proposition \ref{prop.Lp.norm.infinite} applied with $n=k$,  since (as we pointed out) $u_{\alpha,k}$ given by \eqref{counterexample.noninteger.general.2} has support in $\overline{B}_{3/4}^{(k)}$.
\end{proof}

We devote the rest of the section to the proofs of Propositions \ref{prop.Lp.norm.infinite} and \ref{prop.triple.norm.bounded}. In the sequel we assume 
\[
n-1<\lambda<n.
\] 
Recall that
\begin{equation}\label{function.counterexample}
u_{\alpha,n}(x)=\left(\textrm{dist}(x,\mathcal{C}_{n,\lambda})^{-\alpha}-4^\alpha\right)_{+}
\end{equation}
for $\mathcal{C}_{n,\lambda}=\{0\}\times C_{\gamma}\subset\mathbb{R}^{n-1}\times[-1/2,1/2]$, where $C_{\gamma}$ is the generalized Cantor set with parameter $\gamma=1-2^{1-\frac{1}{n-\lambda}}$ defined in the beginning of this section. We have
\begin{equation}\label{eq.formula.complement.cantor}
\left[-\frac12,\frac12\right]\setminus C_{\gamma}:=\bigcup_{l=1}^{\infty}\bigcup_{m=1}^{2^{l-1}}G_{l,m},
\end{equation}
where the union is disjoint and $G_{l,m}$ are the $2^{l-1}$ gap-intervals introduced in generation~$l$, namely\footnote{We will not need the following precise expression for the gaps, but only to understand their size and self-similar structure.}
\begin{equation}\label{GLM}
\begin{split}
G_{l,m}
&=
\left(\frac{1-\gamma}{2}\right)^{l-1}\left(-\frac{\gamma}{2},\frac{\gamma}{2}\right)
 +
h_{l,m}  
\\
&=
\left(
h_{l,m} -\frac{\gamma}{2}\left(\frac{1-\gamma}{2}\right)^{l-1},
h_{l,m}  +\frac{\gamma}{2}\left(\frac{1-\gamma}{2}\right)^{l-1}
\right),
\end{split}
\end{equation}
where 
 \begin{equation}\label{eq.definition.h_{l,m}}
h_{l,m}
=
-\frac12
+
\frac12
\left(\frac{1-\gamma}{2}\right)^{l-1}
 +
\frac{1+\gamma}{2} \,\sum_{j=1}^{l-1}c_j\left(\frac{1-\gamma}{2}\right)^{j-1}.
 \end{equation}
Here
$c_j\in\{0,1\}$ for all $j=1,\ldots,l-1$, and the index $m\in\{1,2,\ldots, 2^{l-1}\}$ runs through all possible choices of the coefficients $c_1c_2\ldots c_{l-1}$; 
see Figure \ref{figure.gencantor} above.

As a consequence of \eqref{eq.formula.complement.cantor}, we have that $C_{\gamma}$ is a compact set of Lebesgue measure~$0$. In addition, the set
$C_{\gamma}$ is self-similar, that is, $C_{\gamma}=S_1 (C_{\gamma}) \cup S_2(C_{\gamma})$ where $S_1(t)=~(1-~\gamma)t/2-(1+\gamma)/4$ and $S_2(t)=(1-\gamma)t/2+(1+\gamma)/4$. Finally,
the Hausdorff dimension of $C_{\gamma}$ is $-\log{2}/\log((1-\gamma)/2))$, see \cite[Theorem 9.3]{Falconer}. 
Notice that we have chosen $\gamma=1-2^{1-\frac{1}{n-\lambda}}$ such that the Hausdorff dimension of $C_{\gamma}$ is $n-\lambda$.

\subsection{Proof of Proposition \ref{prop.Lp.norm.infinite}: Computation of the $L^{q}$ norm}
In this subsection we provide the proof of Proposition \ref{prop.Lp.norm.infinite}.

\begin{proof}[Proof of Proposition \ref{prop.Lp.norm.infinite}] 
Recall that $\mathcal{C}_{n,\lambda}=\{0\}\times C_{\gamma}\subset\mathbb{R}^{n-1}\times[-1/2,1/2].$
We will denote $x=(x',x'')\in\mathbb{R}^{n-1}\times\mathbb{R}$.
Since $C_{\gamma}$ has zero Lebesgue measure, 
 \eqref{eq.formula.complement.cantor} leads to
\[
\begin{split} 
\| u_{\alpha,n}\|_{L^{q}(B_1^{(n)})}^q
&\geq
\int_{-\frac12}^{\frac12}
\int_{B_{1/4}^{(n-1)}}
\left(\textrm{dist}(x,\mathcal{C}_{n,\lambda})^{-\alpha}-4^\alpha\right)_{+}^{q}\,dx'dx''
\\
&=
\sum_{l=1}^{\infty}\sum_{m=1}^{2^{l-1}}
\int_{G_{l,m}}\int_{B_{1/4}^{(n-1)}} \left(\textrm{dist}(x,\mathcal{C}_{n,\lambda})^{-\alpha}-4^\alpha\right)_{+}^{q}\, dx'dx'', 
\end{split}
\]
where  $G_{l,m}$ are given by \eqref{GLM} and \eqref{eq.definition.h_{l,m}}.

 An affine change of variables $x'=\left(\frac{1-\gamma}{2}\right)^{l-1}t'$, $x''=\left(\frac{1-\gamma}{2}\right)^{l-1}t''+h_{l,m}$ and the self-similarity of $\mathcal{C}_{n,\lambda}$ yield
\[
\begin{split}
&\int_{G_{l,m}}\int_{B_{1/4}^{(n-1)}} 
\left(\textrm{dist}(x,\mathcal{C}_{n,\lambda})^{-\alpha}-4^\alpha\right)_{+}^{q}\, dx'dx''
\\
&=
\left(\frac{1-\gamma}{2}\right)^{(n-\alpha q)(l-1)}
\int_{-\frac{\gamma}{2}}^{\frac{\gamma}{2}}\int_{\left\{|t'|\leq\frac14\left(\frac{2}{1-\gamma}\right)^{l-1}\right\}}
\left(\textrm{dist}(t,\mathcal{C}_{n,\lambda})^{-\alpha}-\left(\frac{1-\gamma}{2}\right)^{\alpha \,(l-1)}4^\alpha\right)_{+}^{q}
\, dt'dt''
\\
&\geq
\left(\frac{1-\gamma}{2}\right)^{(n-\alpha q)(l-1)}
\int_{-\frac{\gamma}{2}}^{\frac{\gamma}{2}}\int_{B_{1/4}^{(n-1)}}
\left(\textrm{dist}(t,\mathcal{C}_{n,\lambda})^{-\alpha}-4^\alpha\right)_{+}^{q}\, dt'dt''
\end{split}
\]
for each $m\in\{1,\ldots,2^{l-1}\}$.
Therefore, adding these $2^{l-1}$ integrals of generation $l$, and then summing in $l$, we have
\[
\| u_{\alpha,n}\|_{L^{q}(B_1^{(n)})}^q
\geq
\sum_{l=1}^{\infty}
\left(
2
\left(\frac{1-\gamma}{2}\right)^{n-\alpha q}
\right)^{l-1}
\int_{-\frac{\gamma}{2}}^{\frac{\gamma}{2}}\int_{B_{1/4}^{(n-1)}}
\left(\textrm{dist}(t,\mathcal{C}_{n,\lambda})^{-\alpha}-4^\alpha\right)_{+}^{q}\, dt'dt''.
\]
Since the integral on the right-hand side is positive, it is enough to show that the series diverges. This happens whenever
\[
2\left(\frac{1-\gamma}{2}\right)^{n-\alpha q}\geq1,
\]
which by our choice of $\gamma$ is equivalent to
\[
n-\alpha q\leq\frac{-\log{2}}{\log{\frac{1-\gamma}{2}}}=n-\lambda.
\]
This inequality holds by hypothesis.
\end{proof}

\subsection{Proof of Proposition \ref{prop.triple.norm.bounded}: Bound for the ``triple norm"}

The proof of Proposition \ref{prop.triple.norm.bounded} has two parts. The first one (Lemma \ref{lemma.reductions} below) shows that in order to bound the triple norm of $\nabla u_{\alpha,n}$, it suffices to only consider points
$y\in\{0\}\times[-1/2,1/2]$, instead of the full $\overline{B_1}^{(n)}$.
More precisely, we prove that
\begin{equation}\label{eq.explanation.proof.triple.lambda}
\begin{split}
\opnorm{\nabla u_{\alpha,n}}_{p,\lambda;B_1^{(n)}}^{p} \leq C 
 \sup_{y'=0, |y''|\leq\frac12}\int_{-\frac12}^{\frac12}\int_{B_{1/4}^{(n-1)}} \textnormal{dist}(x,\mathcal{C}_{n,\lambda})^{-\alpha p-p}\,|x-y|^{\lambda-n}\, dx'dx''.
\end{split}
 \end{equation}
 
Since $C_{\gamma}$ has zero Lebesgue measure, by \eqref{eq.formula.complement.cantor} we can write the outer integral on the right-hand side of \eqref{eq.explanation.proof.triple.lambda} as an infinite sum of integrals over the disjoint gap-intervals $G_{l,m}$
 of decreasing size.

The second part of the proof of 
Proposition \ref{prop.triple.norm.bounded},
and crucial point in the argument, is how to estimate these integrals in terms of the size of the gaps in such a way that the series converges. 
Here, there are two cases to be considered according to the position of the singularity $y$ relative to a given gap: the case when $y$ lies on the closure of a gap (and hence the function $x\mapsto|x-y|^{\lambda-n}$ is singular), and the case when the gap is uniformly away from $y$ (and hence $|x-y|^{\lambda-n}$ can be bounded above and factored out from the integral). We deal with these two cases in Lemmas \ref{lemma.integral.distance.and.y} and \ref{lemma.integral.only.distance} respectively.

\begin{lemma}\label{lemma.integral.distance.and.y}
Let 
$
G
=
\left((1-\gamma)/2\right)^{l-1}\left(-\gamma/2,\gamma/2\right)
+
h
$
 for some $l\geq1$ and $h\in\mathbb{R}$ of the form \eqref{eq.definition.h_{l,m}} (i.e. $G$ is a gap-interval introduced in generation $l$). Assume $y\in\{0\}\times \overline{G}$ and $\lambda-\alpha p-p >0$.
Then, we have that
\begin{equation}\label{eq.lemma.integral.distance.to.both.C.and.y.1}
\begin{split}
\int_{G}
\int_{B_{1/4}^{(n-1)}}
 \textnormal{dist}(x,\mathcal{C}_{n,\lambda})^{-\alpha p-p}&\,|x-y|^{\lambda-n}\, dx'dx''
\\
&\leq
 C \,
 \left(
\left(
\frac{1-\gamma}{2}
\right)
^{(\lambda-\alpha p-p)(l-1)}
+
l
\left(
\frac{1-\gamma}{2}
\right)
^{l-1}
\right)
\end{split}
\end{equation}
for a constant $C$ depending only on $n, p,\lambda,$ and $\alpha$.
\end{lemma}

\begin{proof}
Denote $a=\left(\gamma/2\right)\left((1-\gamma)/2\right)^{l-1}$ so that $G=(-a,a)+h$. To relate $\textnormal{dist}(x,\mathcal{C}_{n,\lambda})$ and $|x-y|$ we consider the midpoints between $y''$ and $h+a$, and between $y''$ and $h-a$.
In this way,
we have the bound
\begin{equation}\label{Lemma.y.belongs.to.G.eq1}
\begin{split}
\int_{ G}
\int_{B_{1/4}^{(n-1)}} 
& \textnormal{dist}(x,\mathcal{C}_{n,\lambda})^{-\alpha p-p}\,|x-y|^{\lambda-n}\, dx'dx''
\\
&\leq
\int_{
\frac{y''+h+a}{2}
}
^{
h+a
}
\int_{B_{1/4}^{(n-1)}} 
 \textnormal{dist}(x,\mathcal{C}_{n,\lambda})^{\lambda-n-\alpha p-p}
\, dx'dx''
\\
&\qquad+
\int_{
\frac{y''+h-a}{2}
}
^{
\frac{y''+h+a}{2}
}
\int_{B_{1/4}^{(n-1)}} 
 |x-y|^{\lambda-n-\alpha p-p}
 \, dx'dx''
 \\
&\qquad+
\int_{
h-a
}
^{
\frac{y''+h-a}{2}
}
\int_{B_{1/4}^{(n-1)}} 
 \textnormal{dist}(x,\mathcal{C}_{n,\lambda})^{\lambda-n-\alpha p-p}
 \, dx'dx''=I_1+I_2+I_3.
\end{split}
\end{equation}

Let us estimate $I_1$ first. For this, notice that whenever $x''\in G$, we have
\begin{equation}\label{formula.for.distance.Cantor}
 \textnormal{dist}(x,\mathcal{C}_{n,\lambda})^2
 =
 |x'|^2
 +
 \min
 \left\{
 (h+a-x'')^2, (x''-h+a)^2
 \right\}.
\end{equation}
Therefore,
\[
I_1
=
\int_{
\frac{y''+h+a}{2}
}
^{
h+a
}
\int_{B_{1/4}^{(n-1)}} 
\left(
 |x'|^2
 +
 (h+a-x'')^2
 \right)^\frac{\lambda-n-\alpha p-p}{2}
\, dx'dx''
\]
and a change to cylindrical coordinates and the change of variables $t=h+a-x''$
yield
\begin{equation}\label{Lemma.y.belongs.to.G.I1}
I_1
=
C
\int_{0}^{
\frac{h+a-y''}{2}
}
\int_{0}^{\frac{1}{4}}
\left(
 r^2
 +
 t^2
 \right)^\frac{\lambda-n-\alpha p-p}{2}
r^{n-2}\, drdt.
\end{equation}
Similarly,
\begin{equation}\label{Lemma.y.belongs.to.G.I3}
I_3
=
C
\int_
{
\frac{h-a-y''}{2}
}
^{0}
\int_{0}^{\frac{1}{4}}
\left(
 r^2
 +
 t^2
 \right)^\frac{\lambda-n-\alpha p-p}{2}
r^{n-2}\, drdt.
\end{equation}
On the other hand, since $y'=0$, we have 
\begin{equation}\label{Lemma.y.belongs.to.G.I2}
\begin{split}
I_2
&=
\int_{
\frac{y''+h-a}{2}
}
^{
\frac{y''+h+a}{2}
}
\int_{B_{1/4}^{(n-1)}} 
\left( |x'|^2+(x''-y'')^2\right)^\frac{\lambda-n-\alpha p-p}{2}
 \, dx'dx''
 \\
&=
C 
\int_{
\frac{y''+h-a}{2}
}
^{
\frac{y''+h+a}{2}
}
\int_{0}^\frac{1}{4} 
\left( r^2+(x''-y'')^2\right)^\frac{\lambda-n-\alpha p-p}{2}
r^{n-2} \, drdx''
 \\
&=
C 
\int_{
\frac{h-a-y''}{2}
}
^{
\frac{h+a-y''}{2}
}
\int_{0}^\frac{1}{4} 
\left( r^2+t^2\right)^\frac{\lambda-n-\alpha p-p}{2}
r^{n-2} \, drdt,
\end{split}
\end{equation}
after applying a change to cylindrical coordinates and the change of variables $t=x''-y''$. 

Then, \eqref{Lemma.y.belongs.to.G.eq1} and \eqref{Lemma.y.belongs.to.G.I1}--\eqref{Lemma.y.belongs.to.G.I2}, and the change of variables $z_1=2t-h+y''$, $z_2=2r$ give
\[
\begin{split}
\int_{ G}
\int_{B_{1/4}^{(n-1)}} 
& \textnormal{dist}(x,\mathcal{C}_{n,\lambda})^{-\alpha p-p}\,|x-y|^{\lambda-n}\, dx'dx''
\\
&\leq
2\,C 
\int_{
\frac{h-a-y''}{2}
}
^{
\frac{h+a-y''}{2}
}
\int_{0}^\frac{1}{4} 
\left( r^2+t^2\right)^\frac{\lambda-n-\alpha p-p}{2}
r^{n-2} \, drdt
\\
&=
C \,2^{1-\lambda+\alpha p+p}
\int_{-a}^{a}
\int_{0}^{\frac12} 
\left( (z_1+h-y'')^2+z_2^2\right)^\frac{\lambda-n-\alpha p-p}{2}
 z_2^{n-2} \, dz_2dz_1
 \\
&=
C\, 2^{1-\lambda+\alpha p+p}\, J(y''-h,0),
\end{split}
\]
for
\[
J(y)
=
\int_{-a}^{a}
\int_{0}^{\frac12} 
|z-y|^{\lambda-n-\alpha p-p}
 z_2^{n-2} \, dz_2dz_1.
\]
Here we can apply Lemma \ref{lemma.monotonicity} in dimension 2,
with $\Omega=[-a,a]\times[0,1/2]$, $\theta=n-\lambda+\alpha p+p$, $h(z)=z_2^{n-2}$, and $\eta=0$.
Therefore,   $J$ is non-increasing with respect to $y_1$ in $[0,\infty)$, and non-decreasing with respect to $y_1$ in $(-\infty,0]$. In particular, $J(y''-h,0)\leq J(0)$, with an equality for $y''=h$. We conclude
\begin{equation}\label{Lemma.y.belongs.to.G.ineq.almost.there}
\begin{split}
\int_{ G}
\int_{B_{1/4}^{(n-1)}} 
& \textnormal{dist}(x,\mathcal{C}_{n,\lambda})^{-\alpha p-p}\,|x-y|^{\lambda-n}\, dx'dx''
\\
&\leq
C \,2^{1-\lambda+\alpha p+p}
\int_{-a}^{a}
\int_{0}^{\frac12} 
\left( z_1^2+z_2^2\right)^\frac{\lambda-n-\alpha p-p}{2}
 z_2^{n-2} \, dz_2dz_1
\end{split}
\end{equation}
independently of $y$.

We estimate now the integral on the right-hand side of \eqref{Lemma.y.belongs.to.G.ineq.almost.there} by applying the change of variables $x_1=z_1$, $z_2=|x_1|x_2$
\[
\begin{split}
\int_{-a}^{a}
\int_{0}^{\frac12} &
\left( z_1^2+z_2^2\right)^\frac{\lambda-n-\alpha p-p}{2}
 z_2^{n-2} \, dz_2dz_1
 \\
 &=
 \int_{-a}^{a} |x_1|^{\lambda-\alpha p-p-1}
\int_{0}^{\frac{1}{2|x_1|}} 
\left( 1+x_2^2\right)^\frac{\lambda-n-\alpha p-p}{2} 
 x_2^{n-2} \, dx_2dx_1.
\end{split}
\]
Notice that $|x_1|<a<1/2$, and therefore we can use $1+x_2^2\geq\max\{1,x_2^2\}$ to show that
\[
\begin{split}
\int_{0}^{\frac{1}{2|x_1|}} 
\left(1+x_2^2\right)^\frac{\lambda-n-\alpha p-p}{2} x_2^{n-2}\, dx_2
&\leq
\int_{0}^{1} 
x_2^{n-2}\, dx_2
+
\int_{1}^{\frac{1}{2|x_1|}} 
x_2^{\lambda-\alpha p-p-2}\, dx_2.
\end{split}
\]
Assume first that $\lambda-\alpha p-p\neq1$ and recall that $\lambda-\alpha p-p>0$ by hypothesis.
Then,
\[
\begin{split}
\int_{ G}
\int_{B_{1/4}^{(n-1)}} 
& \textnormal{dist}(x,\mathcal{C}_{n,\lambda})^{-\alpha p-p}\,|x-y|^{\lambda-n}\, dx'dx''
\\
&\leq
C
\int_{
0
}
^{
a
}
x_1^{\lambda-\alpha p-p-1}
\left(
1+x_1^{-\lambda+\alpha p+p+1}
\right)
dx_1
\leq
 C 
 \left(a^{\lambda-\alpha p-p}+a\right).
\end{split}
\]
On the other hand, if $\lambda-\alpha p-p=1$, then
\[
\begin{split}
\int_{ G}
\int_{B_{1/4}^{(n-1)}} 
 \textnormal{dist}(x,\mathcal{C}_{n,\lambda})^{-\alpha p-p}&\,|x-y|^{\lambda-n}\, dx'dx''
\\
&\leq
 C 
\int_{
0
}
^{
a
}
\left(
1
-
\log{2x_1}
\right)
dx_1
\leq
C \left(a+a|\log a|\right)
.
\end{split}
\]
In both cases,
\[
\int_{ G}
\int_{B_{1/4}^{(n-1)}} 
 \textnormal{dist}(x,\mathcal{C}_{n,\lambda})^{-\alpha p-p}\,|x-y|^{\lambda-n}\, dx'dx''
\leq
C \big(a^{\lambda-\alpha p-p}+a(1+|\log a|)\big)
\]
and the lemma is proved. 
\end{proof}

The following lemma will allow us to control the triple norm when $y$ belongs to a gap different from $G$. Notice that the exponents on the right-hand side of the estimate are different in Lemmas \ref{lemma.integral.distance.and.y} and \ref{lemma.integral.only.distance}.
\begin{lemma}\label{lemma.integral.only.distance}
Let 
$
G
=
\left((1-\gamma)/2\right)^{l-1}\left(-\gamma/2,\gamma/2\right)
+
h
$
 for some $l\geq1$ and $h\in\mathbb{R}$ of the form \eqref{eq.definition.h_{l,m}} (i.e. $G$ is a gap-interval introduced in generation $l$). Assume $n-\alpha p-p >0$. Then, we have that
\[
\int_{G}
\int_{B_{1/4}^{(n-1)}} 
 \textnormal{dist}(x,\mathcal{C}_{n,\lambda})^{-\alpha p-p}\, dx'dx''
 \leq
C
\left(
\left(
\frac{1-\gamma}{2}
\right)
^{(n-\alpha p-p)(l-1)}
+l\left(
\frac{1-\gamma}{2}
\right)
^{l-1}
\right)
\]
for a constant $C$ depending only on $n, p,$ and $\alpha$.
\end{lemma}

\begin{proof}
Let $a=(\gamma/2)\left((1-\gamma)/2\right)^{l-1}$ so that $G=(-a,a)+h$.
Using cylindrical coordinates and \eqref{formula.for.distance.Cantor}, we have
\[
\begin{split}
\int_{G}
\int_{B_{1/4}^{(n-1)}} &
 \textnormal{dist}(x,\mathcal{C}_{n,\lambda})^{-\alpha p-p}\, dx'dx''
\\
&=
C
\bigg(
\int_{
h
}
^{
h+a
}
\int_{0}^{\frac14} 
\left(r^2+
\left(h+a-x''\right)^2
\right)^\frac{-\alpha p-p}{2}r^{n-2}\, drdx''
\\
&\qquad+
\int_{
h-a
}
^{
h
}
\int_{0}^{\frac14} 
\left(r^2+
\left(x''-h+a\right)^2
\right)^\frac{-\alpha p-p}{2}r^{n-2}\, drdx''
\bigg)
\\
&=
2C
\int_{
0
}
^{
a
}
\int_{0}^{\frac14} 
\left(r^2+\left(z''\right)^2\right)^\frac{-\alpha p-p}{2}r^{n-2}\, drdz''.
\end{split}
\]
Notice that this integral is of the same type as the one in the right-hand side of \eqref{Lemma.y.belongs.to.G.ineq.almost.there}, taking $\lambda=n$ there. It is now easy to check that one can proceed as in the final part of Lemma \ref{lemma.integral.distance.and.y} (taking $\lambda=n$ there) and complete the proof.
\end{proof}

As mentioned before, the following lemma will also be used in the first part of the proof of Proposition \ref{prop.triple.norm.bounded} in order to show that
to bound the supremum in the definition of $\opnorm{\nabla u_{\alpha,n}}_{p,\lambda;B_1^{(n)}}$,
it suffices to take $y\in\{0\}\times[-1/2,1/2]$.

\begin{lemma}\label{lemma.reductions}
Let $u_{\alpha,n}$ be given by \eqref{function.counterexample}. Then,
 \begin{equation}\label{eq.lemma.reductions}
\begin{split}
\opnorm{\nabla u_{\alpha,n}}_{p,\lambda;B_1^{(n)}}^{p} &= \sup_{y\in \overline{B}_1^{(n)}}
 \int_{B_1^{(n)}} |\nabla u_{\alpha,n}(x)|^{p}\,|x-y|^{\lambda-n}\, dx\\
&\leq C 
 \sup_{y'=0, |y''|\leq\frac12}\int_{-\frac12}^{\frac12}\int_{B_{1/4}^{(n-1)}} \textnormal{dist}(x,\mathcal{C}_{n,\lambda})^{-\alpha p-p}\,|x-y|^{\lambda-n}\, dx'dx''
\end{split}
 \end{equation}
 for some constant $C$ depending only on $n$, $p$, and $\alpha$.
\end{lemma}

We postpone the proof of Lemma \ref{lemma.reductions} until the end of the section and proceed instead with the proof of Proposition \ref{prop.triple.norm.bounded}.
The idea is to ``cluster" the gaps according to their distance from $y$, and then use Lemmas \ref{lemma.integral.distance.and.y} and  \ref{lemma.integral.only.distance}.

\begin{proof}[Proof of Proposition \ref{prop.triple.norm.bounded}.]
As a result of Lemma \ref{lemma.reductions} we can assume that $y'=0$ and $y''
\in[-1/2,1/2]$. To simplify the argument below, by Fatou's lemma
we may assume that $y''$ is not the midpoint of any gap $G_{l,m}$ given by \eqref{GLM} and \eqref{eq.definition.h_{l,m}}.

Each generation $l\geq1$ introduces $2^{l-1}$ gaps $G_{l,m}$, and we can write
\begin{equation}\label{integral.triple.norm.written.in.sums.of.Glm}
\begin{split}
\int_{-\frac12}^{\frac12}\int_{B_{1/4}^{(n-1)}} &\textrm{dist}(x,\mathcal{C}_{n,\lambda})^{-\alpha p-p}\,|x-y|^{\lambda-n}\, dx'dx''\\
&=\sum_{l=1}^{\infty}\sum_{m=1}^{2^{l-1}}
\int_{G_{l,m}}\int_{B_{1/4}^{(n-1)}} \textrm{dist}(x,\mathcal{C}_{n,\lambda})^{-\alpha p-p}\,|x-y|^{\lambda-n}\, dx'dx''.
\end{split}
\end{equation}
Recall that the length of the gap $G_{l,m}$ is $\gamma\left((1-\gamma)/2\right)^{l-1}$.

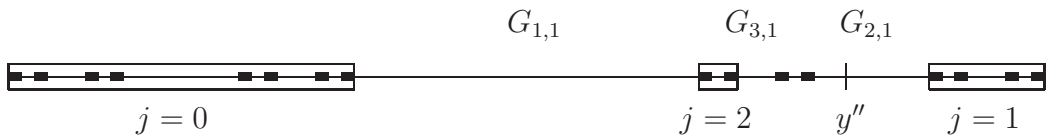
\begin{figure}[t]

\setlength{\unitlength}{.17cm}
\begin{picture}(80,20)(0,-5)

 \put(0,0){\line(1,0){81}}
 
 
 \multiput(0,-1)(54,0){1}{\line(1,0){27}} 
 \multiput(0,-1)(27,0){2}{\line(0,1){2}} 
 \multiput(0,1)(54,0){1}{\line(1,0){27}} 

 \multiput(72,-1)(18,0){1}{\line(1,0){9}} 
 \multiput(72,-1)(9,0){2}{\line(0,1){2}} 
 \multiput(72,1)(18,0){1}{\line(1,0){9}} 

 \multiput(54,-1)(9,0){1}{\line(1,0){3}} 
 \multiput(54,-1)(3,0){2}{\line(0,1){2}} 
 \multiput(54,1)(9,0){1}{\line(1,0){3}} 
 

 \put(65.5,-1){\line(0,1){2}}


 \put(64.7,-4){$y''$}
 
 \put(39,3.5){$G_{1,1}$} 
 \put(65,3.5){$G_{2,1}$} 
 \put(56,3.5){$G_{3,1}$} 

 \put(10,-4){$j=0$} 
 \put(73.5,-4){$j=1$}
 \put(52.5,-4){$j=2$} 
 
 

 \linethickness{3pt}
 
 \multiput(0,0)(2,0){2}{\line(1,0){1}} 
 \multiput(6,0)(2,0){2}{\line(1,0){1}} 
 \multiput(18,0)(2,0){2}{\line(1,0){1}} 
 \multiput(24,0)(2,0){2}{\line(1,0){1}} 
 \multiput(54,0)(2,0){2}{\line(1,0){1}} 
 \multiput(60,0)(2,0){2}{\line(1,0){1}} 
 \multiput(72,0)(2,0){2}{\line(1,0){1}} 
 \multiput(78,0)(2,0){2}{\line(1,0){1}} 

\end{picture}
\caption{Gaps up to generation $l=4$, classified according to
\eqref{eq.definition.indices.clusters}.}
\label{fig.clusters.v2}
\end{figure}

We classify the $2^{l-1}$ gaps of generation $l\geq2$ according to their distance from $y''$ as follows (see Figures \ref{figure.gencantor} and \ref{fig.clusters.v2}):

\begin{enumerate}
 
 \item We split $(-1/2,1/2)$ into two halves, and notice that there are exactly $2^{l-2}$ gaps of generation $l$ in each half. We denote the 
 gaps on the half-interval which does not contain $y''$ by $G_{l,m}^{0}$, with $m=1,2,\ldots,2^{l-2}$. Since half of 
 $G_{1,1}=(-\gamma/2,\gamma/2)$ lies between any of these gaps and $y''$, and the length of $G_{1,1}$ is $\gamma$, we have that
\[
 \textrm{dist}(y'', G_{l,m}^{0})\geq\frac{\gamma}{2}.
\]

\item The  $2^{l-2}$  gaps  remaining from step 1 are contained in an interval $I$ of length $(1-\gamma)/2$. To this interval we apply the procedure in step 1, splitting $I$  into two halves. Notice that there are exactly $2^{l-3}$ gaps of generation $l$ in each half. We denote the gaps on the half-interval that is farthest from $y''$ by $G_{l,m}^{1}$ (recall that $y''$ is not the center of the full interval $I$), with $m=1,2,\ldots,2^{l-3}$. These gaps satisfy
\[
 \textrm{dist}(y'', G_{l,m}^{1})\geq\frac{\gamma}{2}\, \frac{1-\gamma}{2}.
\]


\item Iterating this procedure, at each step $j$ we find exactly $2^{l-2-j}$ gaps of generation~$l$,  denoted by $G_{l,m}^j$  with $m=1,2,\ldots,2^{l-2-j}$. They satisfy 
 \begin{equation}\label{eq.definition.indices.clusters}
 \textrm{dist}(y'', G_{l,m}^j)\geq\frac{\gamma}{2}\left(\frac{1-\gamma}{2}\right)^{j}.
\end{equation}

  \item We continue the iteration until  $j=l-2$, which starts with only two gaps  of generation $l$ left. The farthest from $y''$, denoted by $G_{l,1}^{l-2}$,  satisfies \eqref{eq.definition.indices.clusters} with $j=l-2$.
 On the other hand, we denote the gap closest  to $y''$  by $G_{l,1}$. 
 
\end{enumerate}

Summarizing, among the $2^{l-1}$ gaps of generation $l$ we have selected one, called $G_{l,1}$, in step 4. The gap $G_{l,1}$  is the closest  
to $y''$ among those in generation $l$. The remaining $2^{l-1}-1$ gaps have been clustered into $l-1$ families 
$\{G_{l,m}^0\}_m,\ldots,\{G_{l,m}^j\}_m,\ldots,\{G_{l,m}^{l-2}\}_m$, where the $j$-th family contains $2^{l-2-j}$ gaps of generation $l$ 
which, in addition, satisfy~\eqref{eq.definition.indices.clusters}.

With this classification, we have
\begin{equation}\label{eq.proof.triple.all.sums}
\begin{split}
\sum_{l=1}^{\infty}&\sum_{m=1}^{2^{l-1}}
\int_{G_{l,m}}\int_{B_{1/4}^{(n-1)}} \textrm{dist}(x,\mathcal{C}_{n,\lambda})^{-\alpha p-p}\,|x-y|^{\lambda-n}\, dx'dx''
\\
&=\sum_{l=1}^{\infty}\int_{G_{l,1}}\int_{B_{1/4}^{(n-1)}} \textrm{dist}(x,\mathcal{C}_{n,\lambda})^{-\alpha p-p}\,|x-y|^{\lambda-n}\, dx'dx''
\\
&+
\sum_{l=2}^{\infty}
\sum_{j=0}^{l-2}
\sum_{m=1}^{2^{l-2-j}}
\int_{G_{l,m}^{j}}\int_{B_{1/4}^{(n-1)}} \textrm{dist}(x,\mathcal{C}_{n,\lambda})^{-\alpha p-p}\,|x-y|^{\lambda-n}\, dx'dx''.
\end{split}
\end{equation}

There are two cases to study in the sequel since integrals over $G_{l,1}$ and $G_{l,m}^{j}$ are qualitatively different. The key difference is that 
\eqref{eq.definition.indices.clusters} allows us to control $|x-y|$ from below and the integrals over $G_{l,m}^{j}$ become independent of $y$. Thus, we can apply Lemma \ref{lemma.integral.only.distance} to them. This is not possible for integrals over the gaps $G_{l,1}$. 

For the integrals over a gap $G_{l,1}$, if $y''\in \overline{G}_{l,1}$ then we can apply Lemma \ref{lemma.integral.distance.and.y} directly. 
Instead, if $y''\notin \overline{G}_{l,1}$, we move  $y''$ to the closest   boundary point of $G_{l,1}$ from $y''$, and with this procedure the integral becomes larger (since all  the distances from points in $B_{1/4}^{(n-1)}\times G_{l,1} $ decrease). With this new point  $y''$ the integral can be bounded  using Lemma~\ref{lemma.integral.distance.and.y}.

In fact, when we assume $y''\in \overline{G_{l,1}}$ Lemma \ref{lemma.integral.distance.and.y} yields
\begin{equation}\label{eq.proof.triple.integral.with.y}
\begin{split}
 \sum_{l=1}^{\infty}\int_{G_{l,1}}\int_{B_{1/4}^{(n-1)}} &\,\textrm{dist}(x,\mathcal{C}_{n,\lambda})^{-\alpha p-p}\,|x-y|^{\lambda-n}\, dx'dx''\\
&\leq
C \sum_{l=1}^{\infty} \,
 \left(
\left(
\frac{1-\gamma}{2}
\right)
^{(\lambda-\alpha p-p)(l-1)}
+
l
\left(
\frac{1-\gamma}{2}
\right)
^{l-1}
\right)\leq C,
\end{split}
\end{equation}
uniformly in $y$, since $(1-\gamma)/2<1$ and $\lambda-\alpha p-p>0$ by hypothesis \eqref{conditions.alpha}.

On the other hand, given a gap $G_{l,m}^{j}$, by \eqref{eq.definition.indices.clusters}  we have
\[
\begin{split}
\int_{G_{l,m}^{j}}\int_{B_{1/4}^{(n-1)}}& \textrm{dist}(x,\mathcal{C}_{n,\lambda})^{-\alpha p-p}\,|x-y|^{\lambda-n}\, dx'dx''
\\
&\leq 
C \left(\frac{1-\gamma}{2}\right)^{(\lambda-n)\,j}
\int_{G_{l,m}^{j}}\int_{B_{1/4}^{(n-1)}} \textrm{dist}(x,\mathcal{C}_{n,\lambda})^{-\alpha p-p}\, dx'dx''.
\end{split}
\]
Then, Lemma \ref{lemma.integral.only.distance} leads to
\[
\begin{split}
\int_{G_{l,m}^{j}}\int_{B_{1/4}^{(n-1)}} &\,\textrm{dist}(x,\mathcal{C}_{n,\lambda})^{-\alpha p-p}\,|x-y|^{\lambda-n}\, dx'dx''\\
&\leq 
C
\left(
\left(
\frac{1-\gamma}{2}
\right)
^{(\lambda-n)\,j+(n-\alpha p-p)(l-1)}
+l\left(
\frac{1-\gamma}{2}
\right)
^{(\lambda-n)\,j+l-1}
\right)
\end{split}
\]
uniformly in $m$. Observe that by our choice of $\gamma$, we have $\left(\frac{1-\gamma}{2}\right)^{\lambda-n}=2$ and thus 
\begin{equation}\label{eq.proof.triple.integral.without.y}
\begin{split}
&\sum_{l=2}^{\infty}
\sum_{j=0}^{l-2}
\sum_{m=1}^{2^{l-2-j}}
\int_{G_{l,m}^{j}}\int_{B_{1/4}^{(n-1)}} \textrm{dist}(x,\mathcal{C}_{n,\lambda})^{-\alpha p-p}\,|x-y|^{\lambda-n}\, dx'dx''
\\
&\leq
C
\sum_{l=2}^{\infty}\sum_{j=0}^{l-2}
2^{l-2-j}\left(
\left(
\frac{1-\gamma}{2}
\right)
^{(\lambda-n)\,j+(n-\alpha p-p)(l-1)}
+l\left(
\frac{1-\gamma}{2}
\right)
^{(\lambda-n)\,j+l-1}
\right)
\\
&=
\frac{C}{2}
\sum_{l=2}^{\infty}\sum_{j=0}^{l-2}
\left(
\left(
\frac{1-\gamma}{2}
\right)
^{(\lambda-\alpha p-p)(l-1)}
+l\left(
\frac{1-\gamma}{2}
\right)
^{(\lambda-n+1)\,(l-1)}
\right)
\\
&=
\frac{C}{2}
\sum_{l=2}^{\infty}
\left((l-1)
\left(
\frac{1-\gamma}{2}
\right)
^{(\lambda-\alpha p-p)(l-1)}
+l(l-1)\left(
\frac{1-\gamma}{2}
\right)
^{(\lambda-n+1)\,(l-1)}
\right)\leq C,
\end{split}
\end{equation}
uniformly in $y$, since $(1-\gamma)/2<1$, $\lambda-\alpha p-p>0$, and $\lambda>n-1$ by hypothesis.

Then, \eqref{integral.triple.norm.written.in.sums.of.Glm},
\eqref{eq.proof.triple.all.sums}, 
\eqref{eq.proof.triple.integral.with.y}, and \eqref{eq.proof.triple.integral.without.y} give
\[
\begin{split}
\int_{-\frac12}^{\frac12}\int_{B_{1/4}^{(n-1)}} \textrm{dist}(x,\mathcal{C}_{n,\lambda})^{-\alpha p-p}\,|x-y|^{\lambda-n}\, dx'dx''
\leq
C
\end{split}
\]
uniformly in $y$, and the proof is complete.
\end{proof}

We conclude the article with the proof of Lemma \ref{lemma.reductions}. 
\begin{proof}[Proof of Lemma \ref{lemma.reductions}.]
Note first that the support of $u_{\alpha,n}$, given by \eqref{function.counterexample}, is included in 
$\overline{B}_{1/4}^{(n-1)}\times[-1,1]$. Hence, 
for any given $y\in \overline{B}_{1}^{(n)}$ we have
\[
\begin{split}
 \int_{B_1^{(n)} }& |\nabla u_{\alpha,n}(x)|^{p}\,|x-y|^{\lambda-n}\, dx \leq C\int_{-1}^{1}\int_{B_{1/4}^{(n-1)}} \textrm{dist}(x,\mathcal{C}_{n,\lambda})^{-\alpha p-p}\,|x-y|^{\lambda-n}\, dx'dx'',
 \end{split}
\]
where we have used that the modulus of the gradient of a distance function is equal to $1$ a.e.\ 
Then, our goal is to show that 
\begin{multline*}
 \sup_{y\in \overline{B}_1^{(n)}}\int_{-1}^{1}\int_{B_{1/4}^{(n-1)}} \textrm{dist}(x,\mathcal{C}_{n,\lambda})^{-\alpha p-p}\,|x-y|^{\lambda-n}\, dx'dx''
\\
\leq C
 \sup_{y'=0, |y''|\leq\frac12}\int_{-\frac12}^{\frac12}\int_{B_{1/4}^{(n-1)}} \textnormal{dist}(x,\mathcal{C}_{n,\lambda})^{-\alpha p-p}\,|x-y|^{\lambda-n}\, dx'dx'',
\end{multline*}
from which \eqref{eq.lemma.reductions} follows.

First we will prove that the supremum over $y\in \overline{B}_1^{(n)}$ is bounded by the supremum over the axis, i.e., $y\in\{0\}\times[-1,1]$. Then, that it actually suffices that $|y''|\leq1/2$ instead of $|y''|\leq1$, and finally that it is enough to integrate $x''$ over $[-1/2,1/2]$ instead of the whole $[-1,1]$.
In doing these
we will use twice the monotonicity result in Lemma~\ref{lemma.monotonicity}.

Therefore, consider $y\in \overline{B}_{1}^{(n)}$, and let us show that 
\[
 J_{1}(y)
 :=
 \int_{-1}^{1}
\int_{B_{1/4}^{(n-1)}} 
 \textnormal{dist}(x,\mathcal{C}_{n,\lambda})^{-\alpha p-p}\,|x-y|^{\lambda-n}\, dx'dx''
 \leq J_{1}(0,y''),
\]
where $y=(y',y'')\in \overline{B}_1^{(n)}$, and thus
$y''\in[-1,1]$. In fact, upon a rotation in the $x'$ variables, we can assume that $y'=(y_1,0)\in\mathbb{R}\times\mathbb{R}^{n-2}$ with $y_1\geq0$. Hence, 
we are under the hypotheses of Lemma \ref{lemma.monotonicity} with $\Omega=B_{1/4}^{(n-1)}\times[-1,1]$, 
$h(z)=\textnormal{dist}(z,\mathcal{C}_{n,\lambda})^{-\alpha p-p}$ (which is non-increasing with respect to $z_1$ in $\{z_1\geq0\}$),
 $\theta=n-\lambda$, and $\eta=0$.
Therefore, Lemma \ref{lemma.monotonicity} gives that $ J_{1}$ is non-increasing with respect to $y_1$ in $[0,\infty)$. Hence we conclude that $J_{1}(y)\leq J_{1}(0,y'')$, as desired.

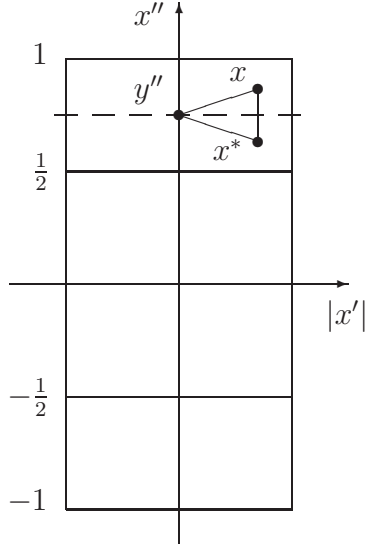
\begin{figure}[t]
\setlength{\unitlength}{1.5cm}
\begin{picture}(6,5.5)(-3,-2.5)
\put(0,0){\vector(1,0){1.5}}
\put(0,0){\line(-1,0){1.5}}
\put(0,0){\vector(0,1){2.5}}
\put(0,0){\line(0,-1){2.3}}

\put(1,-2){\line(0,1){4}} 
\put(-1,-2){\line(0,1){4}} 
\put(-1,2){\line(1,0){2}} 
\put(-1,-2){\line(1,0){2}} 
\put(-1,1){\line(1,0){2}} 
\put(-1,-1){\line(1,0){2}} 

\multiput(-1.1,1.5)(.33,0){7}{\line(1,0){.2}} 

\put(0,1.5){\line(3,1){.7}}
\put(0,1.5){\line(3,-1){.7}}
\put(.7,1.26){\line(0,1){.48}}

\put(.7,1.26){\circle*{.1}} 
\put(.7,1.73){\circle*{.1}} 
\put(0,1.5){\circle*{.1}} 

\put(-.4,2.3){$x''$}
\put(1.3,-.35){$|x'|$}
\put(.45,1.8){$x$}
\put(.3,1.1){$x^*$}
\put(-.4,1.65){$y''$}

\put(-1.3,1.95){$1$}
\put(-1.3,0.93){$\frac{1}{2}$}
\put(-1.52,-1.07){$-\frac{1}{2}$}
\put(-1.52,-2){$-1$}
\end{picture}
\caption{Monotonicity argument for $J_2$ in the proof of Lemma \ref{lemma.reductions}.}
\label{figure.monotonicity}
\end{figure}

A similar argument shows that we just need to consider the case $|y''|\leq1/2$ instead of $|y''|\leq1$. 
In fact, define 
\[
 J_{2}(y''):=
\int_{-1}^{1}
\int_{B_{1/4}^{(n-1)}} 
 \textnormal{dist}(x,\mathcal{C}_{n,\lambda})^{-\alpha p-p}\,
 \Big(|x'|^2+(x''-y'')^2\Big)^{\frac{\lambda-n}{2}}
 \, dx'dx''
\]
and assume $|y''|\geq1/2$. By symmetry, we can assume $y''\geq1/2$. 
In order to apply Lemma \ref{lemma.monotonicity}, we take $x''$ as the direction $e_1$ which is ``privileged" in the lemma, while the rest of the hypotheses are fulfilled for  
$(z'',z')\in\Omega=[-1,1]\times B_{1/4}^{(n-1)}$, 
$h(z'',z')=\textnormal{dist}((z',z''),\mathcal{C}_{n,\lambda})^{-\alpha p-p}$ (which is non-increasing with respect to $z''$ in $\{z''\geq1/2\}$),
 $\theta=n-\lambda$, and $\eta=1/2$, since the set $\mathcal{C}_{n,\lambda}$
is contained in $\{0\}\times[-1/2,1/2]$
(see Figure \ref{figure.monotonicity}).
Then, Lemma \ref{lemma.monotonicity} gives that $ J_{2}(y'')$ is non-increasing with respect to $y''$ in $[1/2,\infty)$, and therefore it is enough to study the case $y''\leq1/2$. The case $y''\leq-1/2$ follows similarly by taking $\eta=-1/2$ in the lemma.

Finally, let $y=(0,y'')$ with $|y''|\leq1/2$.
Notice that for every $x\in B_{1/4}^{(n-1)}\times[1/2,1]$ and its reflected $x^*$ with respect to $\{x''=1/2\}$, we have $|x^*-y|\leq|x-y|$ and $ \textnormal{dist}(x^*,\mathcal{C}_{n,\lambda})\leq \textnormal{dist}(x,\mathcal{C}_{n,\lambda})$, which give 
\[
\begin{split}
 \textnormal{dist}(x,\mathcal{C}_{n,\lambda})^{-\alpha p-p}\,|x-y|^{\lambda-n}\leq
 \textnormal{dist}(x^*,\mathcal{C}_{n,\lambda})^{-\alpha p-p}\,|x^*-y|^{\lambda-n}.
\end{split}
\]
This leads to
\[
\begin{split}
\int_{\frac12}^{1}\int_{B_{1/4}^{(n-1)}} & \textrm{dist}(x,\mathcal{C}_{n,\lambda})^{-\alpha p-p}\,|x-y|^{\lambda-n}\, dx'dx''
\\
&\leq 
\int_{0}^{\frac{1}{2}}\int_{B_{1/4}^{(n-1)}} \textrm{dist}(x,\mathcal{C}_{n,\lambda})^{-\alpha p-p}\,|x-y|^{\lambda-n}\, dx'dx'',
\end{split}
\]
and similarly for the integral over $B_{1/4}^{(n-1)}\times[-1,-\frac12]$.
Therefore,
\[
\begin{split}
\int_{-1}^{1}
&\int_{B_{1/4}^{(n-1)}} 
 \textnormal{dist}(x,\mathcal{C}_{n,\lambda})^{-\alpha p-p}\,
 \Big(|x'|^2+(x''-y'')^2\Big)^{\frac{\lambda-n}{2}}
 \, dx'dx'' 
 \\
 \leq&\,
2 \int_{-\frac12}^{\frac12}
\int_{B_{1/4}^{(n-1)}} 
 \textnormal{dist}(x,\mathcal{C}_{n,\lambda})^{-\alpha p-p}\,
 \Big(|x'|^2+(x''-y'')^2\Big)^{\frac{\lambda-n}{2}}
 \, dx'dx'',
 \end{split}
\]
which completes the proof of the lemma.
\end{proof}

\medskip

\noindent{\bf Acknowledgements:} The first author would like to thank Joan Orobitg and Joan Verdera for a stimulating discussion on the topic of this paper. The authors also thank Giuseppe Mingione for interesting comments and for bringing 
\cite{Adams.Lewis.1982,Bensoussan-Frehse,Fonseca-Maly-Mingione} to their 
attention after the completion of this manuscript, 
as well as the referee for some appropriate remarks and references \cite{Hou-Xiao, Talenti.1976}.


\bibliographystyle{plain}

\end{document}